\documentclass[12pt,dvipsnames]{amsart}
\usepackage[colorlinks=true, pdfstartview=FitV, linkcolor=blue, citecolor=blue, urlcolor=blue]{hyperref}
\usepackage{geometry,pb-diagram}
\usepackage{amssymb, amscd, amsmath, amsfonts, amsthm}
\usepackage{stmaryrd}
\usepackage{comment}
\usepackage{multicol}
\usepackage{MnSymbol}
\usepackage{enumerate}

\excludecomment{moredetails}

\usepackage{xargs}
\usepackage[colorinlistoftodos,prependcaption,textsize=tiny,linecolor=yellow,backgroundcolor=yellow!25,bordercolor=yellow]{todonotes}

\renewcommand{\todo}[1]{\vspace{5 mm}\par \noindent
\marginpar{\textsc{ToDo}}
\framebox{\begin{minipage}[c]{.99 \textwidth}
\tt #1 \end{minipage}}\vspace{5 mm}\par}

\def\unprotectedboldentry#1{\textcolor{Red}{\textbf{#1}}}
\def\boldentry{\protect\unprotectedboldentry}
\usepackage{tikz}
\usetikzlibrary{calc}
\usepackage{ifthen}
\newcommand{\tikztableau}[2][scale=0.6,every node/.style={font=\small}]{
    \def\newtableau{#2}
    \begin{array}{c}
    \begin{tikzpicture}[#1]
    \coordinate (x) at (-0.5,0.5);
    \coordinate (y) at (-0.5,0.5);
    \foreach \row in \newtableau {
        \coordinate (x) at ($(x)-(0,1)$);
        \coordinate (y) at (x);
        \draw ($(y)+(0.5,0.5)$) rectangle +(0,-1);
        \foreach \entry in \row {
            \ifthenelse{\equal{\entry}{X}}
            {
                \node (y) at ($(y) + (1,0)$) {};
                \fill[color=gray!10] ($(y)-(0.5,0.5)$) rectangle +(1,1);
                \draw[color=gray] ($(y)-(0.5,0.5)$) rectangle +(1,1);
            }
            {
                \ifthenelse{\equal{\entry}{\boldentry X}}
                {
                    \node (y) at ($(y) + (1,0)$) {};
                    \fill[color=gray] ($(y)-(0.5,0.5)$) rectangle +(1,1);
                    \draw ($(y)-(0.5,0.5)$) rectangle +(1,1);
                 }
                 {
                    \node (y) at ($(y) + (1,0)$) {\entry};
                    \draw ($(y)-(0.5,0.5)$) rectangle +(1,1);
                 }
            }
            }
        }
    \end{tikzpicture}
    \end{array}}
\newcommand{\tikztableausmall}[1]{\tikztableau[scale=0.45,every node/.style={font=\rm\small}]{#1}}
\newcommand{\tikztableautiny}[1]{\tikztableau[scale=0.35,every node/.style={font=\rm\tiny}]{#1}}
\newcommand{\tikztableaunano}[1]{\tikztableau[scale=0.30,every node/.style={font=\rm\tiny}]{#1}}
 
\def\ZZ{\mathbb{Z}}

\def\BB{\mathbb{B}}

\def\la{{\lambda}}

\def\Des{\operatorname{Des}}

\def\op{{[}}
\def\opp{{[}} 
\def\cl{{]}}
\def\clp{{]}} 

\def\read{\operatorname{\rm{read}}}
\def\sym{\operatorname{\mathsf{Sym}}}

\def\Qsym{\operatorname{\mathsf{QSym}}}

\def \fS{{\mathfrak S}}
\def \HH{{H}}
\def\Nsym{\operatorname{\mathsf{NSym}}}

\def\Ys{{Y}}
\def\sort{\operatorname{sort}}
\def\Tal{{\mathfrak{T}_\alpha^\lambda}}
\newcommand{\polytope}{I(\alpha,\beta,\nu)}

\newtheorem{Theorem}{Theorem}[section]
\newtheorem{Proposition}[Theorem]{Proposition}
\newtheorem{Corollary}[Theorem]{Corollary}
\newtheorem{Lemma}[Theorem]{Lemma}
\newtheorem{Example}[Theorem]{Example}
\theoremstyle{definition}
\newtheorem{Remark}[Theorem]{Remark}
\newtheorem{Definition}[Theorem]{Definition}
\newtheorem{Conjecture}[Theorem]{Conjecture}

\begin{document}

\title[Multiplicative structures of the immaculate basis]{Multiplicative structures of the immaculate basis of non-commutative symmetric functions}
\author[C. Berg \and N. Bergeron \and F. Saliola \and L. Serrano \and M. Zabrocki]{Chris Berg \and Nantel Bergeron \and Franco Saliola \and Luis Serrano \and Mike Zabrocki}
\date{\today}
 
\begin{abstract}
We continue our development of a new basis for the algebra of non-commutative symmetric functions.
This basis is analogous to the Schur basis for the algebra of symmetric functions, and it shares many of its wonderful properties.
For instance, in this article we describe non-commutative versions of the Littlewood-Richardson rule and the Murnaghan-Nakayama rule. 
A surprising relation develops among non-commutative Littlewood-Richardson coefficients, which has implications to the commutative case. Finally, we interpret these new coefficients geometrically as the number of integer points inside a certain polytope.
\end{abstract}

\maketitle
\setcounter{tocdepth}{3}
 
\section{Introduction} The Schur functions appear throughout mathematics:\ as the representatives for the Schubert
classes in the cohomology of the Grassmannian; as the characters for the irreducible representations of
the symmetric group and of the general linear group; and as an orthonormal basis for the algebra of symmetric functions.
The ubiquitousness of the Schur basis makes it an object of central importance in the theory of symmetric functions.
Among the characteristic properties of Schur functions 
there are combinatorial formulae using tableaux,
orthogonality relations, algebraic formulae and each of these
may be taken either as the relation defining the basis or a consequence of the definition.

If the algebra of symmetric functions, $\sym$, is viewed as an algebra with one
commutative generator at each degree, then the algebra of non-commutative symmetric
functions \cite{GKLLRT}, $\Nsym$, is an analogous algebra with one non-commutative
generator at each degree.
Until relatively recently, 
the only clear proposal for elements of $\Nsym$ that are analogous to the Schur functions
was the ribbon basis.

While this is a natural proposition, researchers
have begun to question the assumption that the ribbon basis is the best possible Schur-analogue
for $\Nsym$ and have proposed other bases (each with a different rule for
computing the commutative image) \cite{BBSSZ, HLMvW11, CFLSX}.  Each are worth exploring as
potential tools for resolving some of the positivity, combinatorial and
representation theoretical questions in symmetric functions.  For instance, Hall-Littlewood
symmetric functions seem to have interesting analogues using these bases \cite{HLMvW11,BBSSZ}
and these are a potential avenue for answering open questions about $q,t$-Kostka \cite{Mac}
and generalized Kostka coefficients \cite{SW}.

In \cite{BBSSZ}, the authors proposed a basis, called the immaculate basis, 
for the algebra of $\Nsym$.
This basis is analogous to the Schur basis of $\sym$ through a Jacobi-Trudi-like defining formula as well as
a combinatorial formula using composition tableaux, and because the immaculate functions
indexed by a partition project onto the Schur functions indexed by the same partition under the natural
map from $\Nsym$ to $\sym$.
In \cite{BBSSZ2}, the authors constructed indecomposable modules for the $0$-Hecke 
algebra whose characters correspond to the immaculate functions under the duality 
between $\Nsym$ and $\Qsym$. This provides further evidence of the virtuousness of the immaculate basis.
 
The goal of this paper is to further develop the immaculate basis of $\Nsym$. The paper is organized as follows. In Section~\ref{catfish} we review the symmetric function theory that we intend to emulate in $\Nsym$. In Section~\ref{oldcookie} we present a brief introduction to $\Nsym$ and the immaculate basis. In Section~\ref{bananapeel} we describe an extension of the Pieri rule for immaculate functions (Theorem \ref{thm:Pieri}) to ribbon shapes (Theorem \ref{slime}). In Section~\ref{sec:mnrule} we prove a version of the Murnaghan-Nakayama rule (Theorem \ref{waffle}) for the immaculate functions. In Section~\ref{sec:dualmnrule} we deduce a dual version of our Murnaghan-Nakayama rule (Corollary 6.2). In Section~\ref{sec:LRimm} we formulate and prove an analogue of the Littlewood-Richardson rule for immaculate functions (Theorem \ref{ImmLRrule}) and discover a relation amongst Littlewood-Richardson coefficients for Schur functions (Corollary \ref{cor:newsym}) that is not
easily deduced from the theory of symmetric functions. Finally, in Section~\ref{niceass} we provide a geometric interpretation of the immaculate Littlewood-Richardson coefficients as the number of integer points inside certain polytopes.

 \subsection{Acknowledgements}
This work is supported in part by  NSERC and FRQNT grants.
It is partially the result of a working session at the Algebraic
Combinatorics Seminar at the Fields Institute with the active
participation of C. Benedetti, C. Ceballos
(especially for his help on Corollary~\ref{cor:hives}),  
J. S\'anchez-Ortega, O. Yacobi, E. Ens,  H. Heglin, D. Mazur and T. MacHenry.
 
This research was facilitated by computer exploration using the open-source
mathematical software \texttt{Sage}~\cite{sage} and its algebraic
combinatorics features developed by the \texttt{Sage-Combinat}
community~\cite{sage-co}.

The authors are very grateful to Darij Grinberg for a meticulous reading of an
earlier version of this article, which led to numerous improvements in the
text.

\section{Symmetric Function Background}\label{catfish}
In this section, we  build 
notation in order to state the classical version of the results we emulate on the immaculate basis of $\Nsym$.
For more details, we refer the reader to \cite{Mac}, \cite{Sagan}, or \cite{Sta}.
\subsection{Partitions}
 
A \textit{partition} of a non-negative integer $n$ is a sequence of positive integers
$\lambda = \op \lambda_1, \lambda_2, \dots, \lambda_m \cl$ which sum to $n$ and satisfy
$\lambda_1 \geq \lambda_2 \geq \cdots \geq \lambda_m$; it is denoted
$\lambda \vdash n$. Partitions are of central importance to
algebraic combinatorics; among other things, partitions of $n$ index a basis
for the symmetric functions of degree $n$.
As such, the notation and terminology for partitions is fairly standard, thus we adhere to the notation in the aforementioned references.

\subsection{Symmetric functions}
We let $\sym$ denote the ring of symmetric functions.
As an algebra, $\sym$ is the commutative algebra over $\mathbb{Q}$ freely generated by elements
$\{h_1, h_2, \dots\}$. The algebra $\sym$ is graded, where $h_i$ has degree $i$. A natural basis for the degree $n$ component of $\sym$ is formed by the complete homogeneous symmetric functions of degree $n$, defined by
$h_\lambda := h_{\lambda_1} h_{\lambda_2} \cdots h_{\lambda_m}$.
The algebra $\sym$ is usually identified with a subalgebra of the polynomial algebra
$\mathbb{Q}\llbracket x_1, x_2, \dots\rrbracket$.
Under this identification, $h_i$ corresponds to the sum of all monomials of degree $i$.
 
\subsection{Schur functions}
\label{ssec:schurfunctions}
We define the basis of \emph{Schur functions} via the relationship to the complete homogeneous basis $\{ h_\lambda : \lambda \vdash n\}$
specified by the \emph{Jacobi-Trudi formula}:
for a sequence of integers
$\alpha = \op \alpha_1, \alpha_2, \ldots, \alpha_\ell\cl \vdash n$, we define
\begin{equation}\label{sweatpotato} s_\alpha := \det \begin{bmatrix}
h_{\alpha_1}&h_{\alpha_1+1}&\cdots&h_{\alpha_1+\ell-1}\\
h_{\alpha_2-1}&h_{\alpha_2}&\cdots&h_{\alpha_2+\ell-2}\\
\vdots&\vdots&\ddots&\vdots\\
h_{\alpha_\ell-\ell+1}& h_{\alpha_\ell-\ell+2}&\cdots&h_{\alpha_\ell}\\
\end{bmatrix} =
\det \left[h_{\alpha_i + j - i}\right]_{1 \leq i,j \leq \ell},
\end{equation}
where we use the convention that $h_0 = 1$ and $h_{-m} = 0$ for $m>0$.  In particular
the Schur functions indexed by partitions form a basis for the ring of symmetric functions.

\subsection{Power sum symmetric functions} The \emph{power sum} symmetric function $p_k$ is defined as:
\[p_k = \sum_i x_i^k.\] The collection $\{ p_\lambda:= p_{\lambda_1}p_{\lambda_2}\cdots p_{\lambda_m}\}_{\lambda \vdash n}$
forms a basis for the degree $n$ component of $\sym$.
 
\subsection{The Murnaghan-Nakayama rule}
The \emph{Murnaghan-Nakayama rule} is an explicit combinatorial expansion for the product of a Schur function with a power sum $p_k$.
For two partitions $\lambda$ and $\mu$ with $\mu \subset \lambda$, we say that $\lambda/\mu$
is a \emph{border strip} if the diagram consisting of $\lambda$ without $\mu$ is connected and contains no $2 \times 2$ square.
The \emph{height} of a border strip is defined as the number of rows it contains.
\begin{Theorem}
For $k>0$ and a partition $\mu$, \[ s_\mu p_k = \sum (-1)^{ht(\lambda/\mu)-1} s_\lambda,\]
the sum over partitions $\lambda$ for which $\lambda/\mu$ is a border strip of size $k$.
\end{Theorem}
 
\begin{Example}
For $\lambda = \op 2,2,2 \cl$ and $k = 3$:
\[ s_{222} p_3 = s_{222111} - s_{22221} + s_{333} - s_{432} + s_{522}.\]
 \end{Example}
 
\subsection{The Littlewood-Richardson rule}
The \emph{Littlewood-Richardson rule} is an explicit combinatorial expansion for the product of two Schur functions.
 
A word $w$ over the alphabet $\{1, 2, \dots \}$
is said to be \textit{Yamanouchi} if in every prefix of $w$,
the number of occurrences of $j$ is greater than or equal to the number of occurrences of $j+1$, for all $j$.
The \textit{reading word} of a skew tableau is the word formed by reading each row from right to left, starting from the top row and moving down.

\begin{Theorem}\label{LRrule}
For partitions $\lambda$ and $\mu$,
\[ s_\lambda s_\mu = \sum_\nu c_{\lambda, \mu}^\nu s_\nu,\]
where $\nu$ is a partition of $|\lambda| + |\mu|$ and $c_{\lambda,\mu}^\nu$ is the number of skew tableaux of shape
$\nu/\lambda$ whose reading word is a Yamanouchi word of content $\mu$.
\end{Theorem}
 
\begin{Example} We give an example with $\mu = \lambda = \op 2,1 \cl$.
\[
\begin{array}{cccccccccccc}
s_{21} & \hspace{-.2in} \cdot \hspace{-.2in} &s_{21} & \hspace{-.2in}=\hspace{-.2in} & s_{2211} &\hspace{-.2in}+\hspace{-.2in}& s_{222} &\hspace{-.2in}+\hspace{-.2in}&\,s_{3111}\\[.1in]
\tikztableausmall{{\boldentry X, \boldentry X},{\boldentry X}}& &\tikztableausmall{{1,1},{2}}  &   &
\tikztableausmall{{\boldentry X,\boldentry X},{\boldentry X,1},{1},{2}} & &
\tikztableausmall{{\boldentry X, \boldentry X},{\boldentry X,1},{1,2}} & &
\tikztableausmall{{\boldentry X, \boldentry X, 1},{\boldentry X},{1}, {2}} & &
\end{array}
\]
\[
\begin{array}{ccccccccccccc}
& & &\hspace{-.2in}+\hspace{-.2in}&\,\ 2s_{321} &\hspace{-.2in}+\hspace{-.2in}&\, s_{33}&\hspace{-.2in}+\hspace{-.2in}&\, s_{411}&\hspace{-.2in}+\hspace{-.2in}&\, s_{42}\\[.1in]
& & & &
 \tikztableausmall{{\boldentry X,\boldentry X,1},{\boldentry X,1},{2}}
 \tikztableausmall{{\boldentry X,\boldentry X,1},{\boldentry X,2},{1}}  & &
 \tikztableausmall{{\boldentry X,\boldentry X, 1},{\boldentry X,1,2}} & &
 \tikztableausmall{{\boldentry X, \boldentry X, 1,1},{\boldentry X}, {2}} & &
 \tikztableausmall{{\boldentry X,\boldentry X, 1, 1},{\boldentry X,2}} & &
\end{array}
\]
\end{Example}
 
\section{The non-commutative symmetric functions}\label{oldcookie}
 
In this paper we study graded Hopf algebras whose $n$-th homogeneous component
admits a basis indexed by compositions of $n$.
In this section, we review standard definitions and notations concerning
compositions and define the Hopf algebra of non-commutative symmetric
functions and the Hopf algebra of quasi-symmetric functions.

\subsection{Compositions}
 
A \textit{composition} of a non-negative integer $n$ is a list of positive integers
$\alpha = \op \alpha_1, \alpha_2, \dots, \alpha_m \cl$ which sum to $n$. It is denoted by $\alpha \models n$.
The entries $\alpha_i$ of the composition are the \emph{parts}
of the composition.  The \emph{size} of the composition is the sum of the parts
and will be denoted $|\alpha|:=n$.  The \emph{length} of the composition is the
number of parts and will be denoted $\ell(\alpha):=m$.
(More generally, the length of any list $L$ will be denoted by $\ell(L)$.)
We let $\sort(\alpha)$
denote the partition obtained by sorting the terms in $\alpha$.

Compositions of $n$ correspond to subsets of $\{1, 2, \dots, n-1\}$.
We will follow the convention of identifying $\alpha =
\op\alpha_1, \alpha_2, \dots, \alpha_m\cl$ with the subset \[\mathcal{D}(\alpha) =
\{\alpha_1, \alpha_1+\alpha_2, \alpha_1+\alpha_2 + \alpha_3, \dots, \alpha_1+\alpha_2+\dots + \alpha_{m-1} \}.\]

If $\alpha$ and
$\beta$ are both compositions of $n$, we say that $\alpha \leq  \beta$ in \emph{refinement order} if
$\mathcal{D}(\beta) \subseteq \mathcal{D}(\alpha)$. For instance, we have that $\op1,1,2,1,3,2,1,4,2\cl \leq \op4,4,2,7\cl$, since
$\mathcal{D}(\opp 1,1,2,1,3,2,1,4,2 \clp) = \{1,2,4,5,8,10,11,15\}$ and $\mathcal{D}(\opp 4,4,2,7 \clp) = \{4,8,10\}$.

For two compositions $\alpha \models n$ and $\beta \models m$,
define $\op \alpha, \beta \cl$ to be the composition of $n + m$
obtained by concatenating $\alpha$ and $\beta$; in other words,
\begin{gather*}
    \op \alpha, \beta \cl
    = \op \alpha_1, \alpha_2, \dots, \alpha_{\ell(\alpha)},
    \beta_1, \beta_2, \dots, \beta_{\ell(\beta)} \cl.
\end{gather*}

\subsection{Non-commutative symmetric functions}
The algebra of \emph{non-commutative symmetric functions} $\Nsym$ is the free (non-commutative) algebra generated by the elements $H_1, H_2, \dots$.
It is graded, where the generator $H_i$ has degree $i$.
Therefore a basis for the degree $n$ component of $\Nsym$ is the collection of complete homogeneous non-commutative functions
$H_\alpha := H_{\alpha_1} H_{\alpha_2} \cdots H_{\alpha_m}$ for $\alpha = \op\alpha_1, \alpha_2, \dots, \alpha_m\cl$ a composition of $n$.
There exists a projection defined on the generators by
\begin{eqnarray*}
\chi: \Nsym 	& \longrightarrow 	& \sym \\
H_i			& \longmapsto		& h_i
\end{eqnarray*}
which extends multiplicatively, so that $\chi(H_\alpha) = h_{\sort(\alpha)}$.
 
The algebra $\Nsym$ is isomorphic to the Grothendieck ring of the finitely generated projective representations of the
$0$-Hecke algebra (see \cite{KrTh, DKLT}).
Under this isomorphism, the isomorphism classes of the projective indecomposable modules are identified with the \textit{ribbon functions}.
The ribbon functions $R_\alpha$ can be described in terms of the complete homogeneous non-commutative functions:
\[
R_\alpha = \sum_{\beta \geq \alpha} (-1)^{\ell(\alpha)-\ell(\beta)} \HH_\beta
\hskip .2in \textrm{or equivalently} \hskip .2in \HH_\alpha = \sum_{\beta \geq \alpha} R_\beta.
\]
Under the projection $\chi$, the ribbon functions are mapped onto the ribbon (skew) Schur functions
(see \cite{GKLLRT} for more details). They also admit a simple product formula: if $m, l \geq 1$, then
\begin{equation}\label{productribbon}
    R_{\op\alpha_1, \dots, \alpha_m\cl}
    R_{\op\beta_1, \dots, \beta_l\cl}
    =
    R_{\op\alpha_1, \dots, \alpha_m, \beta_1, \dots, \beta_l\cl}
    +
    R_{\op\alpha_1, \dots, \alpha_{m-1}, \alpha_m + \beta_1, \beta_2, \dots, \beta_l\cl}.
\end{equation}

The reference \cite{GKLLRT} introduces two bases of $\Nsym$ that are analogous
to the power sum symmetric functions.
We focus here on the $\Psi$ basis, which we define via its relationship to the ribbon basis: for a non-negative integer $k$,
\begin{equation}\label{psi}\Psi_k = \sum_{i=0}^{k-1} (-1)^i R_{\op1^i, k-i\cl} =
k H_k - \sum_{i=1}^{k-1} H_i \Psi_{k-i}\end{equation}
and $\Psi_\alpha = \Psi_{\alpha_1} \cdots \Psi_{\alpha_m}$ for compositions $\alpha = \op\alpha_1, \dots, \alpha_m\cl$. It is known that $\chi(\Psi_k) = p_k$.

There is a Hopf algebra structure on $\Nsym$ with coproduct defined on the generators as
\begin{equation}\label{eq:coprodHn}
\Delta(H_r) = \sum_{i=0}^r H_i \otimes H_{r-i},
\end{equation}
and on the basis $H_\alpha$ by $\Delta( H_\alpha ) = \Delta(H_{\alpha_1}) \Delta( H_{\alpha_2}) \cdots \Delta(H_{\alpha_m})$.
The elements $\Psi_k$ are primitive with respect to the coproduct:
that is, $\Delta(\Psi_k) = 1 \otimes \Psi_k + \Psi_k \otimes 1$.

\subsection{Quasi-symmetric functions}
The (graded) dual Hopf algebra to $\Nsym$ is the Hopf algebra of quasi-symmetric functions, $\Qsym$.
As a vector space, $\Qsym$ is the subspace of $\mathbb{Q}\llbracket x_1, x_2, \dots \rrbracket$
spanned by the \emph{monomial quasi-symmetric functions} $M_\alpha$ defined by:
\begin{equation}
  \label{monomial-qsym}
  M_\alpha = \sum_{i_1 < i_2 < \cdots < i_m} x_{i_1}^{\alpha_1} x_{i_2}^{\alpha_2} \cdots x_{i_m}^{\alpha_m}.
\end{equation}
The duality between $\Nsym$ and $\Qsym$ is established by the pairing
\[
\begin{array}{cccccccccc}
\langle \cdot, \cdot \rangle & : & \Nsym \times \Qsym 	& \longrightarrow	& \mathbb{Q} \\
& & \langle H_\alpha, M_\beta \rangle 				& = 				& \delta_{\alpha,\beta}.
\end{array}
\]
As a consequence, the products and coproducts of $\Nsym$ and $\Qsym$ are related by the following identities. For $F, G \in \Nsym$ and $K, L \in \Qsym$,
\begin{equation}\label{eq:coprod}
\langle F \otimes G, \Delta_{\Qsym}(K) \rangle = \langle  F \cdot G, K \rangle
\end{equation}
and
\begin{equation}\label{eq:prod}
\langle F, K \cdot L \rangle = \langle \Delta_{\Nsym}(F), K \otimes L \rangle.
\end{equation}
When it is clear from the context, we use $\Delta$ to denote the coproduct without specifying whether the algebra is $\Nsym$, $\Qsym$ or $\sym$.

\subsection{Immaculate functions}
In \cite{BBSSZ}, we defined a new basis of $\Nsym$, called the \emph{immaculate basis}.
It is readily defined in terms of the complete basis via an analogue of the
classical Jacobi-Trudi formula (see Section \ref{ssec:schurfunctions}).
For any sequence $\alpha = \op\alpha_1, \alpha_2, \dots, \alpha_m\cl \in \ZZ^m$
with $m \geq 0$, define
\begin{equation}\label{skunk}
\fS_\alpha  = \sum_{\sigma \in S_m} (-1)^\sigma \HH_{\op\alpha_1+\sigma_1 -1, \alpha_2 + \sigma_2 -2, \dots, \alpha_m + \sigma_m - m\cl},
\end{equation}
where $\HH_0 = 1$ and $\HH_i = 0$ if $i < 0$.
Thus, $\fS_\alpha$ is defined for any sequence of integers;
but only those with $\alpha$ a composition are part of the immaculate basis.
To see that this is indeed a basis, one proves that the change of basis matrix
relating these elements with the complete basis is unitriangular.
It follows from this definition that
$\chi (\fS_\lambda) = s_\lambda$
for every partition $\lambda$.

Equivalently, the immaculate functions can be constructed using certain
\emph{creation operators} (this is the definition of the immaculate
functions given in \cite{BBSSZ}).
In analogy with the Bernstein operators \cite{Mac,Zel}, which are creation
operators for the Schur functions, we define for $m \in \ZZ$,
\begin{equation}\label{operdef}
\BB_{m} = \sum_{i \geq 0} (-1)^i H_{m+i} M_{\op1^i\cl}^\perp~,
\end{equation}
 where $M_{\op1^i\cl}^\perp$ denotes the adjoint operator to multiplication by $M_{\op1^i\cl} = F_{\op1^i\cl}$ in $\Qsym$.
 
\begin{Proposition}[{\cite[Definition 3.2 and Theorem 3.27]{BBSSZ}}]
For $\alpha \in \ZZ^m$,
\begin{equation}
    \left(\BB_{\alpha_1} \circ \BB_{\alpha_2} \circ \cdots \circ \BB_{\alpha_m}\right)
    \left( 1 \right) = \fS_\alpha.
\end{equation}
\end{Proposition}

\section{Product of an immaculate function with a ribbon function}\label{bananapeel}
 
In this section, we prove that the product of an immaculate function with a ribbon function has a positive expansion in the immaculate basis.
 
\subsection{Skew immaculate tableaux}

For two compositions $\alpha$ and $\beta$ with $\alpha_i \geq \beta_i$ for all $i$, we construct the skew diagram $\alpha / \beta$ as follows. First we superimpose the upper left corners of $\alpha$ and $\beta$, and then we remove the cells of $\beta$ from the resulting diagram.

A \emph{composition tableau} is a map from the cells of $\alpha/\beta$ to ${\mathbb N}$. (Note that there is more than one definitions of ``composition tableau'' in the literature.)
The \emph{shape} $\alpha/\beta$ of the tableau is denoted $sh(T)$.  The \emph{content} of the tableau is the integer vector whose $i^{th}$ entry is
equal to the number of $i$'s appearing in the tableau, and it is denoted $c(T)$. The \emph{reading word} of the tableau, denoted $\read(T)$,
is the word formed by reading the entries in each row from right to left, starting from the top row and moving down.

A composition tableau is called \emph{immaculate} if the entries in
each row are weakly increasing from left to right and the entries in the first column are strictly increasing from top to bottom.

An immaculate tableau of shape $\alpha/\beta$ is \emph{standard} if each of the numbers 1 through $|\alpha|-|\beta|$ appears exactly once.

The \emph{descent set} of a standard immaculate tableau is the set of all
positive integers $j$ that appear in a row strictly above the row containing
$j+1$.
Its \emph{descent composition} is the composition associated with the descent
set under the map $\mathcal{D}$; it is denoted $D(T)$.
 
\begin{Example}\label{duck} A standard skew immaculate tableau of shape $\op2,3,2,1\cl/\op 1,2\cl$ is
$$T = \tikztableausmall{{\boldentry X,2},{\boldentry X,\boldentry X,4}, {1,3},{5}}~.$$
Its reading word is $\read(T) = 24315$.
Its descent set is $\{ 2, 4 \}$
since $2$ appears above the row containing $3$
and $4$ appears above the row containing $5$;
its descent composition  is $D(T) = \op 2,2,1\cl$.
\end{Example}
 
\subsection{The product rule}
 
In \cite[Theorem 3.5]{BBSSZ}, we established an analogue of the classical Pieri rule for expanding the product $ \fS_\alpha H_i$ in the immaculate basis.

\begin{Theorem}[Right-Pieri rule]\label{thm:Pieri}
For a composition $\alpha$ and an integer $i$,
\[\fS_\alpha  \HH_i  =  \sum_{  \alpha  \subset_{i}  \gamma}  \fS_\gamma,\]
where $\alpha  \subset_{i}  \gamma$  if:
\begin{enumerate}
\item $|\gamma| = |\alpha| + i$,
\item $\alpha_j \leq \gamma$ for all $1 \leq j \leq \ell(\alpha)$,
\item $\ell(\gamma) \leq \ell(\alpha) + 1.$
\end{enumerate} 
\end{Theorem}
 
The Pieri rule is an instance of the following more general result (since $R_i = H_i$).
 
\begin{Theorem}\label{slime}
For compositions $\alpha$ and $\beta$,
\[ \fS_\alpha R_\beta = \sum_{\substack{sh(T)= \gamma / \alpha \\ D(T) = \beta}} \fS_\gamma,\]
namely, the sum runs over all standard immaculate tableaux of shape $\gamma/\alpha$ and descent composition~$\beta$.
\end{Theorem}
\begin{proof}
We prove this by induction on the length of $\beta$. If $\ell(\beta) = 1$ then this is just the Pieri rule.
Suppose $\beta$ is of length $m$. The product rule for ribbons, Equation \eqref{productribbon},
implies that
\[
H_{\beta_1} R_{\op\beta_2, \beta_3, \dots, \beta_m\cl} = R_{\beta_1} R_{\op\beta_2, \beta_3, \dots, \beta_m\cl }
= R_{\op\beta_1+\beta_2, \beta_3, \beta_4, \dots, \beta_m\cl} + R_\beta.
\]
 
By the Pieri rule (Theorem \ref{thm:Pieri}) and induction, we have
\begin{equation}\label{snickers}
\fS_\alpha H_{\beta_1}R_{\op\beta_2, \beta_3, \dots, \beta_m\cl} =
\sum_{\substack{   sh(S) = \eta /  \alpha,  sh(P) = \gamma / \eta\\ D(S) = \op\beta_1 \cl, D(P) = \op\beta_2, \dots, \beta_m\cl}}
\fS_\gamma. 
\end{equation}
Again by induction,
\[
\fS_\alpha R_{\op\beta_1+\beta_2, \beta_3, \beta_4, \dots, \beta_m\cl}
= \sum_{\substack{sh(Q) =\delta / \alpha \\ D(Q) = \op\beta_1+\beta_2, \beta_3, \dots, \beta_m\cl}} \fS_\delta.
\]
Consider the summation in Equation \eqref{snickers}:
the tableaux obtained by combining $S$ and the tableau obtained by adding
$\beta_1$ to the entries of $P$ either have descent set $\beta$ or
$\op\beta_1+\beta_2, \beta_3, \dots, \beta_m\cl$.
Thus,
\[
\fS_\alpha R_\beta = \fS_\alpha(H_{\beta_1}R_{\op\beta_2, \beta_3, \dots, \beta_m\cl} - 
R_{\op\beta_1+\beta_2, \beta_3, \beta_4, \dots, \beta_m\cl} )
= \sum_{\substack{sh(T) =  \gamma / \alpha \\ D(T) = \beta}} \fS_\gamma. \qedhere
\]
\end{proof}
 
\begin{Example} For $\alpha = \op2,1\cl$ and $\beta = \op1,2\cl$.
\[
\begin{array}{cccccccccccccccc}
\fS_{21} R_{12} & \hspace{-.2in}=\hspace{-.2in} & \fS_{2112} &\hspace{-.2in}+\hspace{-.2in}& \fS_{2121} &\hspace{-.2in}+\hspace{-.2in}&\,\fS_{2211}&\hspace{-.2in}+\hspace{-.2in}&\,\fS_{222}&\hspace{-.2in}+\hspace{-.2in}&\, \fS_{231}&\hspace{-.2in}+\hspace{-.2in}&\, \fS_{3111}\\[0.1in]
&&
\tikztableausmall{{\boldentry X, \boldentry X},{\boldentry X},{1}, {2,3}} & &
\tikztableausmall{{\boldentry X, \boldentry X},{\boldentry X}, {1,3}, {2}} & &
\tikztableausmall{{\boldentry X, \boldentry X},{\boldentry X,3}, {1}, {2}} & &
\tikztableausmall{{\boldentry X, \boldentry X},{\boldentry X,1}, {2,3}} & &
\tikztableausmall{{\boldentry X, \boldentry X},{\boldentry X,1,3}, {2}} & &
\tikztableausmall{{\boldentry X, \boldentry X,3},{\boldentry X},{1}, {2}} & &
\end{array}
\]
\[
\begin{array}{ccccccccccccc}
\hspace{-.2in}+\hspace{-.2in}&\,\ 2\fS_{321} &\hspace{-.2in}+\hspace{-.2in}&\, \fS_{312}&\hspace{-.2in}+\hspace{-.2in}&\, \fS_{33}&\hspace{-.2in}+\hspace{-.2in}&\, \fS_{411}&\hspace{-.2in}+\hspace{-.2in}&\, \fS_{42}\\[0.1in]
&
 \tikztableausmall{{\boldentry X, \boldentry X,1},{\boldentry X,3}, {2}}
 \tikztableausmall{{\boldentry X, \boldentry X,3},{\boldentry X,1}, {2}} & &
 \tikztableausmall{{\boldentry X, \boldentry X,1},{\boldentry X}, {2,3}} & &
 \tikztableausmall{{\boldentry X, \boldentry X,1},{\boldentry X,2,3}} & &
 \tikztableausmall{{\boldentry X, \boldentry X,1,3},{\boldentry X}, {2}} & &
   \tikztableausmall{{\boldentry X, \boldentry X,1,3},{\boldentry X,2}}
\end{array}
\]
\end{Example}
 
\section{Murnaghan-Nakayama rule for immaculate functions}\label{sec:mnrule}

The classical Murnaghan-Nakayama rule provides the Schur expansion of the
product of a Schur function and a power sum.
We formulate and prove a rule for the immaculate expansion of the product of an
immaculate function $\fS_\alpha$ and a noncommutative power sum $\Psi_k$.
We call this rule a Murnaghan-Nakayama rule for immaculate functions.  The
commutative image of this rule can be used to derive the rule for the
expansion of $s_\lambda p_k$ using the usual Murnaghan-Nakayama rule
by `straightening' terms which are not indexed
by partitions.

\begin{Theorem}\label{waffle}
For a composition $\alpha$ and a positive integer $k$,
\[
\fS_\alpha \Psi_k = - \sum_{\beta \vDash k} (-1)^{\ell(\beta)} \fS_{\op\alpha_1 \ldots, \alpha_{\ell(\alpha)},\beta_1, \ldots, \beta_{\ell(\beta)}\cl} + \sum_{j=1}^{\ell(\alpha)} \fS_{\op\alpha_1, \dots, \alpha_j+k,\dots, \alpha_{\ell(\alpha)}\cl}.
\]
\end{Theorem}

This rule follows from Lemmas \ref{fishtaco} and \ref{puppy} below.  We first provide
an example that might help
to visualize the diagrams that arise in the expression since the coefficient of
$\fS_\alpha$ in a product of $\Psi_\beta$ could be interpreted as a set of composition
tableaux which record the successive shapes that occur in the expressions for $\Psi_{\beta_1},
(\Psi_{\beta_1})\Psi_{\beta_2},((\Psi_{\beta_1})\Psi_{\beta_2})\Psi_{\beta_3}, \ldots, \Psi_{\beta}$.

\begin{Example} We illustrate Theorem \ref{waffle} with $\alpha = \op1,3,2\cl$ and $k=3$.
\[
\begin{array}{cccccccccccccccccccccc}
\fS_{132}\Psi_3 & = & \fS_{132111} & - & \fS_{13212} & - & \fS_{13221} & + & \fS_{1323}\\[0.05in]
& &
\tikztableautiny{{\boldentry X},{\boldentry X, \boldentry X, \boldentry X},{\boldentry X, \boldentry X},{\null}, {\null},{\null}} & &
\tikztableautiny{{\boldentry X},{\boldentry X, \boldentry X, \boldentry X},{\boldentry X, \boldentry X},{\null}, {\null,\null}} & &
\tikztableautiny{{\boldentry X},{\boldentry X, \boldentry X, \boldentry X},{\boldentry X, \boldentry X},{\null,\null}, {\null}} & &
\tikztableautiny{{\boldentry X},{\boldentry X, \boldentry X, \boldentry X},{\boldentry X, \boldentry X},{\null,\null,\null}}\\[0.45in]
& + & \fS_{135} & + & \fS_{162} & + & \fS_{432}. \\[0.05in]
&&\tikztableautiny{{\boldentry X},{\boldentry X, \boldentry X, \boldentry X},{\boldentry X, \boldentry X,,,}} & &
\tikztableautiny{{\boldentry X},{\boldentry X, \boldentry X, \boldentry X,,,},{\boldentry X, \boldentry X}} & &
\tikztableautiny{{\boldentry X,,,},{\boldentry X, \boldentry X, \boldentry X},{\boldentry X, \boldentry X}} & & \\ 
\end{array}
\]
\end{Example}

There is a non-commutative bias that is built into this rule since
in general there no bijection between the terms in the immaculate-expansion of
$\Psi_{\alpha}$ and $\Psi_{[\alpha_{\sigma(1)}, \alpha_{\sigma(2)}, \ldots, \alpha_{\sigma(\ell(\alpha))}]}$
for $\sigma \in S_{\ell(\alpha)}$.  A small example for comparison shows that there are 12 terms
appearing in the expansion of $\Psi_{[3,2]}$ and 11 terms appearing in
the expansion of $\Psi_{[2,3]}$.  The rule for the product $\Psi_k \fS_\alpha$
in general has more terms than $\fS_\alpha \Psi_k$ 
and is not modeled as neatly by composition tableaux (consider for instance the special case
implied by Theorem \ref{ImmLRrule}).

\begin{Lemma}\label{fishtaco}
For a non-commutative symmetric function $f$, a positive integer $k$,
and a non-negative integer $m$,
\[
\BB_m( f )\Psi_k = \BB_m(f \Psi_k) + \BB_{m+k}(f)~.
\]
\end{Lemma}

\begin{proof}
We use well known Hopf algebra relations to develop the actions of the operators used
in these formulas.  From \cite[Lemma 2.5]{BBSSZ} we know that
\begin{equation}
M_{\op1^r\cl}^\perp(H_k) = \begin{cases}
H_k&\text{if }r=0,\\
H_{k-1}&\text{if }r=1,\\
0&\text{otherwise.}
\end{cases}
\end{equation}
Using \cite[Lemma 2.4]{BBSSZ} and the coproduct formula $\Delta(M_{\op1^r\cl}) = \sum_{i=0}^r
M_{\op1^i\cl} \otimes M_{\op1^{r-i}\cl}$ we can compute by induction on Equation \eqref{psi}
to conclude
\begin{equation}
M_{\op1^r\cl}^\perp(\Psi_k) = \begin{cases}
\Psi_k&\text{if }r=0,\\
(-1)^{k-1}&\text{if }r=k\text{ and }k>0,\\
0&\text{otherwise.}
\end{cases}
\end{equation}
We may now calculate the
action of $\BB_m$ on the product $f\Psi_k$ where $f \in \Nsym$:
\begin{gather}
\BB_m(f \Psi_k) = \sum_{i \geq 0} (-1)^i H_{m+i} \sum_{r=0}^i M_{\op1^{i-r}\cl}^\perp(f) M_{\op1^r\cl}^\perp(\Psi_k) = \BB_m(f) \Psi_k - \BB_{m+k}(f)~.
\qedhere
\end{gather}
\end{proof}

\begin{Lemma}\label{puppy}
For $k \geq 1$,
\[
\Psi_k = \sum_{\beta \models k} (-1)^{\ell(\beta)+1} \fS_\beta~.
\]
\end{Lemma}
\begin{proof}
This follows by using Theorem \ref{slime} in the ribbon expansion of
$\Psi_k$ from Equation \eqref{psi}.  There is exactly one standard immaculate tableau
of shape $\beta \models k$ of length $j$ with descent composition equal to 
$\op1^{j-1},k-j+1\cl$.  That tableau
has entries $1,2,\ldots,j$ down the first column with $j+1, j+2,\ldots,k$ placed in
such a way as to not create any other descents.  We use this interpretation to
conclude that
\begin{equation}\label{toenails}
R_{\op1^{j-1},k-j+1\cl} = \sum_{\substack{\beta \models k\\\ell(\beta) = j}} \fS_{\beta}~.\qedhere
\end{equation}
\end{proof}

\begin{proof}[Proof of Theorem \ref{waffle}]
Theorem \ref{waffle} follows by inductively applying Lemma \ref{fishtaco} to the 
expression $\fS_\alpha \Psi_k = \BB_{\alpha_1}( \BB_{\alpha_2}( \cdots
\BB_{\alpha_{\ell(\alpha)}}(1)\cdots))\Psi_k$ and using Lemma \ref{puppy}
as the base case.
\end{proof}

\begin{Corollary}\label{chocolatewaffle}
For a composition $\alpha$ and a positive integer $k$,
\[
\fS_\alpha \Psi_k
= \sum_{j=1}^{k} \fS_{\op\alpha_1, \dots, \alpha_{\ell(\alpha)}, \begin{tiny}\underbrace{0, 0, \dots, 0}_{j-1}, k \end{tiny}\cl}
+ \sum_{j=1}^{\ell(\alpha)} \fS_{\op\alpha_1, \alpha_2, \dots, \alpha_j+k, \dots, \alpha_{\ell(\alpha)}\cl}.
\] 
\end{Corollary}

\begin{proof} This expression will follow by showing
that $\fS_{\op0^{j-1},k\cl} = (-1)^{j-1} R_{\op 1^{j-1},k\cl}$.
Once this is done, Equation \eqref{toenails} implies that
\[
\Psi_k = \sum_{j=1}^k \fS_{\op0^{j-1},k\cl}~.
\]
The statement of this Corollary then follows by inductively applying Lemma \ref{fishtaco}
to the expression $\fS_\alpha \Psi_k = \BB_{\alpha_1}( \BB_{\alpha_2}( \cdots
\BB_{\alpha_{\ell(\alpha)}}(1)\cdots))\Psi_k$ with this formula as the
base case.

We have from \cite[Lemma 2.5]{BBSSZ} for non-negative integers $i$ and $j$,
\[
M_{\op 1^i\cl}^\perp (R_{\op 1^j\cl}) =
\begin{cases}
R_{\op 1^{j-i}\cl}&\hbox{ if }i\leq j\\
0&\hbox{ if }i>j~.
\end{cases}
\]
A direct computation shows that for $j>0$ that $\BB_0(R_{\op1^j\cl}) = 0$
and for $j\geq0$, $\BB_1(R_{\op1^j\cl}) = R_{\op1^{j+1}\cl}$.
We also have from \cite[Lemma 3.4]{BBSSZ} that
$$\BB_m(f\HH_s) = \BB_m(f)\HH_s - \BB_{m+1}(f) \HH_{s-1}.$$
Now the following expansion of a ribbon follows from Equation \eqref{productribbon}:
\[
R_{\op 1^{j-1},k-j+1\cl} = R_{\op 1^{j-1}\cl} \HH_{k-j+1} -
R_{\op 1^{j-2}\cl} \HH_{k-j+2} + \cdots + (-1)^{j-1} H_{k}~.
\]
Apply $\BB_0$ to both sides of this equation
and simplify the result to conclude that
$$\BB_0\left(R_{\op 1^{j-1},k-j+1\cl}\right) = - R_{\op 1^j,k-j\cl}.$$
We then conclude by induction that $\fS_{\op0^{j-1},k\cl}
= (-1)^{j-1} R_{\op 1^{j-1},k\cl}$.
\end{proof}

\begin{Remark}
Note that Corollary \ref{chocolatewaffle} is a non-commutative analogue of
\cite[Section I.3, Example 11, (1), p. 48]{Mac} which is used in
as an intermediate step in
the proof of the Murnaghan-Nakayama rule.
\end{Remark}

\begin{Example} We graphically
illustrate Corollary \ref{chocolatewaffle}
with $\alpha = \op1,3,2\cl$ and $k=3$.
\[
\begin{array}{cccccccccccccccccccccc}
\fS_{132}\Psi_3 & \!\!\!\!\!\!\!=\!\!\!\!\!\!\! & \fS_{132003} & \!\!\!\!\!\!\!+\!\!\!\!\!\!\! & \fS_{13203} & \!\!\!\!\!\!\!+\!\!\!\!\!\!\! & \fS_{1323} & \!\!\!\!\!\!\!+\!\!\!\!\!\!\! & \fS_{135} & \!\!\!\!\!\!\!+\!\!\!\!\!\!\! & \fS_{162} & \!\!\!\!\!\!\!+\!\!\!\!\!\!\! & \fS_{432}. \\[0.05in]
& &
\tikztableaunano{{\boldentry X},{\boldentry X, \boldentry X, \boldentry X},{\boldentry X, \boldentry X},{},{},{\null,\null,\null}} & &
\tikztableaunano{{\boldentry X},{\boldentry X, \boldentry X, \boldentry X},{\boldentry X, \boldentry X},{},{\null,\null,\null}} & &
\tikztableaunano{{\boldentry X},{\boldentry X, \boldentry X, \boldentry X},{\boldentry X, \boldentry X},{\null,\null,\null}}
& &\tikztableaunano{{\boldentry X},{\boldentry X, \boldentry X, \boldentry X},{\boldentry X, \boldentry X,,,}} & &
\tikztableaunano{{\boldentry X},{\boldentry X, \boldentry X, \boldentry X,,,},{\boldentry X, \boldentry X}} & &
\tikztableaunano{{\boldentry X,,,},{\boldentry X, \boldentry X, \boldentry X},{\boldentry X, \boldentry X}} & & \\ 
\end{array}
\]
\end{Example}

\section{The dual Murnaghan-Nakayama rule for immaculate functions}\label{sec:dualmnrule}

The Hopf algebra of symmetric functions is a self-dual Hopf algebra, and under
this duality the Schur basis is self dual.
Therefore, in addition to describing multiplication by $p_k$ on Schur
functions, the classical Murnaghan-Nakayama rule also describes
the adjoint operator to multiplication by $p_k$, denoted $p_k^\perp$.

The adjoint operators of $\Nsym$ are of the form $F^\perp$ where $F \in \Qsym$
and in $\Qsym$ we have $M_r=p_r$.
Here we formulate a dual version of the Murnaghan-Nakayama rule for immaculate
functions, which describes the action of $M_r^\perp$ on an immaculate function.

\begin{Theorem} \label{th:skewbyM}  For a sequence of integers $\alpha$
    and a composition $\beta$,
    $$M_\beta^\perp \fS_\alpha = \sum_{\gamma} \fS_\gamma$$
    where the sum is over all sequences of integers $\gamma$ of length
    $\ell(\alpha)$ such that removing the zeros from $\alpha - \gamma$ yields
    the composition $\beta$.
\end{Theorem}

\begin{proof}
Recall that in \cite[Lemma 2.4]{BBSSZ} we give a Hopf algebra computation that can be used to compute the commutation
relation between $G^\perp$ for $G \in \Qsym$ and multiplication by $f$ in $\Nsym$.   In particular, since $\Delta(M_\beta)
= \sum_{\op\gamma,\tau\cl = \beta} M_\gamma \otimes M_\tau$, and $M_r^\perp(H_m) = H_{m-r}$, and $M_\beta^\perp( H_m) = 0$ when $\ell(\beta)>1$,
we have that
\begin{gather}
    \label{eqn:MperpHCommutation}
    M_\beta^\perp \circ H_m = H_{m-\beta_1}\circ M_{\op \beta_2, \ldots, \beta_{\ell(\beta)}\cl}^\perp + H_m \circ M_\beta^\perp.
\end{gather}
Now we can use this relation to compute the commutation relation between $G^\perp$ and the creation operator for the immaculate basis.
\begin{align*}
M_\beta^\perp \circ \BB_m &= \sum_{d\geq0} (-1)^d M_\beta^\perp \circ H_{m+d} M_{\op1^{d}\cl}^\perp\\
&= \sum_{d\geq0} (-1)^d (H_{m+d-\beta_1} \circ M_{\op\beta_2, \ldots, \beta_{\ell(\beta)}\cl}^\perp + H_{m+d} \circ M_\beta^\perp) M_{\op1^{d}\cl}^\perp\\
&= \BB_{m-\beta_1}\circ M_{\op\beta_2, \ldots, \beta_{\ell(\beta)}\cl}^\perp + \BB_m \circ M_\beta^\perp~.
\end{align*}
The result follows by induction since
$\fS_\alpha = \BB_{\alpha_1} \BB_{\alpha_2} \cdots \BB_{\alpha_{\ell(\alpha)}}(1)$ and $M_r^\perp(1) =0$.
\end{proof}

As a consequence of this computation, the dual Murnaghan-Nakayama rule is
a special case since $p_r = M_r$.
\begin{Corollary}\label{cor:MrPerpAndImmaculate}
If $\alpha$ is a sequence of integers and $r>0$, then
\[
M_r^\perp \fS_\alpha = \sum_{i = 1}^{\ell(\alpha)} \fS_{\op\alpha_1,\ldots,\alpha_i-r,\ldots,\alpha_{\ell(\alpha)}\cl}.
\]
\end{Corollary}

\begin{moredetails}
We offer a direct proof of Corollary \ref{cor:MrPerpAndImmaculate}
pointed out to us by Darij Grinberg.

\begin{proof}[Direct proof of Corollary \ref{cor:MrPerpAndImmaculate}]
    Taking $\beta = [r]$ in Equation~\eqref{eqn:MperpHCommutation},
    we have that
    \begin{gather*}
        M_r^\perp \circ H_m = H_{m-r} + H_m \circ M_r^\perp,
    \end{gather*}
    where $H_m$ represents ``multiplication by $H_m$''.
    Hence, we have that $M_r$ is a derivation, since for any $X \in \Nsym$,
    \begin{gather*}
        M_r^\perp (H_m X) = M_r^\perp(H_m) X + H_m M_r^\perp(X).
    \end{gather*}
    The result now follows by applying this property
    recursively to Equation~\eqref{skunk}:
    \begin{align*}
        M_r^\perp(\fS_\alpha)
        & = \sum_{i = 1}^{\ell(\alpha)}
            \sum_{\sigma \in S_{\ell(\alpha)}} (-1)^\sigma
            \HH_{\alpha_1 + \sigma_1 - 1}
            \cdots
            M_r^\perp(\HH_{\alpha_i + \sigma_i - i})
            \cdots
            \HH_{\alpha_{\ell(\alpha)} + \sigma_{\ell(\alpha)} - m}
        \\ & = \sum_{i = 1}^{\ell(\alpha)}
            \sum_{\sigma \in S_{\ell(\alpha)}} (-1)^\sigma
            \HH_{\alpha_1 + \sigma_1 - 1}
            \cdots
            \HH_{(\alpha_i - r) + \sigma_i - i}
            \cdots
            \HH_{\alpha_{\ell(\alpha)} + \sigma_{\ell(\alpha)} - m}
        \\ & = \sum_{i = 1}^{\ell(\alpha)}
            \fS_{\op\alpha_1,\ldots,\alpha_i-r,\ldots,\alpha_{\ell(\alpha)}\cl}.
            \qedhere
    \end{align*}
\end{proof}
\end{moredetails}

Note that in the last result, if $r < \alpha_i$ for all $1 \leq i \leq \ell(\alpha)$, we obtain an immaculate-positive expansion. If not, we need to potentially consider immaculate functions
indexed by integer sequences.

\section{The product of immaculate functions}\label{sec:LRimm}
 
In \cite{BBSSZ}, we conjectured that the product of two immaculate functions expands positively in the
immaculate basis when the composition indexing the right immaculate function is a partition.
The right Pieri rule is a special case of this product.
A combinatorial rule for the expansion of the left Pieri rule 
that explains the occurrence of the negative signs is handled in \cite{BSOZ,SXL}.

The expansion of $\fS_\alpha \fS_\lambda$, where $\alpha$ is a composition and
$\lambda$ is a partition, provides new
insight on the multiplication of Schur functions (see for example Corollary \ref{cor:newsym}).
This rule has interesting combinatorics and 
a beautiful geometric interpretation as seen in Section~\ref{niceass}.

The main goal of this section is to give an explicit positive combinatorial rule for the product of two immaculate functions 
 when the composition indexing the right immaculate function is a partition. Our proof follows some techniques developed in \cite{BG90} and
\cite{BS02}.

The general rule for the multiplication of two immaculate functions is out of the scope of this paper. 
In \cite{SXL}, S.~X.~Li extended our involution to explain the left Pieri rule,
but this extension does not provide the necessary insight to provide an expression for the
general product of two immaculate functions indexed by compositions.
 
\subsection{The combinatorial formula for the structure coefficients}
\begin{Definition}
A skew immaculate tableau $T$ will be called \emph{Yamanouchi} if the reading word $\read(T)$ of $T$ is Yamanouchi.
For compositions $\alpha$, $\beta$ and a partition $\lambda$ such that $|\alpha| + |\lambda| = |\beta|$
we let $C_{\alpha, \lambda}^\beta$, the \emph{immaculate Littlewood--Richardson coefficient}, denote the number
of Yamanouchi immaculate tableaux of shape $\beta / \alpha$ and content~$\lambda$.
\end{Definition}
 
\begin{Example} \label{ex:smallex}
Let $\alpha = \op 1,2\cl$, $\lambda = \op 3,2,1\cl$ and $\beta = \op 3,4,2\cl$.
The tableaux drawn below are the only Yamanouchi immaculate tableaux, so $C_{\alpha,\lambda}^\beta = 2$.
$$\tikztableausmall{{\boldentry X,1,1},{\boldentry X,\boldentry X,1,2},{2,3}}
\hspace{.4in}\tikztableausmall{{\boldentry X,1,1},{\boldentry X,\boldentry X,2,2},{1,3}}$$
\end{Example}

\begin{Theorem}\label{ImmLRrule}
For a composition $\alpha$ and a partition $\lambda$, \[ \fS_\alpha \fS_\lambda = \sum_{\beta \models |\alpha|+|\lambda|}
C_{\alpha, \lambda}^\beta \fS_\beta.\]
\end{Theorem}

\begin{Example} \label{ex:prod1} We give an example with $\alpha = \op1,2\cl$ and $\lambda = \op 2,1\cl$.
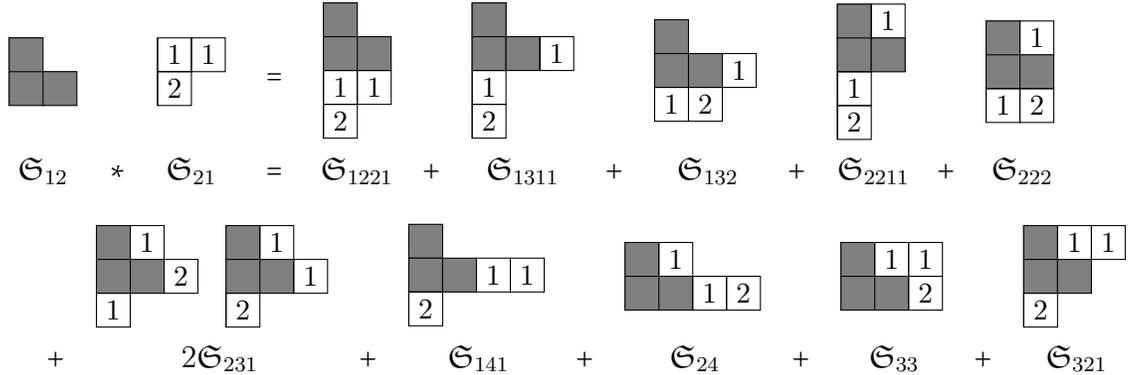
\[
\begin{array}{ccccccccccccc}
\tikztableausmall{{\boldentry X},{\boldentry X,\boldentry X}}& &\tikztableausmall{{1,1},{2}}  & =   &
\tikztableausmall{{\boldentry X},{\boldentry X,\boldentry X},{1,1},{2}} & &
\tikztableausmall{{\boldentry X},{\boldentry X,\boldentry X,1},{1},{2}} & &
\tikztableausmall{{\boldentry X},{\boldentry X,\boldentry X,1},{1,2}} & &
\tikztableausmall{{\boldentry X,1},{\boldentry X,\boldentry X},{1},{2}} & &
\tikztableausmall{{\boldentry X,1},{\boldentry X,\boldentry X},{1,2}} \\
\fS_{12} & \hspace{-.2in}*\hspace{-.2in} &\fS_{21} & \hspace{-.2in}=\hspace{-.2in} & \fS_{1221} &\hspace{-.2in}+\hspace{-.2in}& \fS_{1311} &\hspace{-.2in}+\hspace{-.2in}&\,\fS_{132}&\hspace{-.2in}+\hspace{-.2in}&\,\fS_{2211}&\hspace{-.2in}+\hspace{-.2in}&\, \fS_{222}
\end{array}
\]
\[
\begin{array}{ccccccccccccc}
& & & &
 \tikztableausmall{{\boldentry X,1},{\boldentry X,\boldentry X,2},{1}}
 \tikztableausmall{{\boldentry X,1},{\boldentry X,\boldentry X,1},{2}} & &
 \tikztableausmall{{\boldentry X},{\boldentry X,\boldentry X,1,1},{2}} & &
 \tikztableausmall{{\boldentry X,1},{\boldentry X,\boldentry X,1,2}} & &
 \tikztableausmall{{\boldentry X,1,1},{\boldentry X,\boldentry X,2}} & &
  \tikztableausmall{{\boldentry X,1,1},{\boldentry X,\boldentry X},{2}} \\
& & &\hspace{-.2in}+\hspace{-.2in}&\,\ 2\fS_{231} &\hspace{-.2in}+\hspace{-.2in}&\, \fS_{141}&\hspace{-.2in}+\hspace{-.2in}&\, \fS_{24}&\hspace{-.2in}+\hspace{-.2in}&\, \fS_{33}&\hspace{-.2in}+\hspace{-.2in}&\, \fS_{321}
\end{array}
\]
\end{Example}

\begin{proof}[Proof of Theorem \ref{ImmLRrule}]
Let $m = \ell(\lambda)$.
Using \eqref{skunk}, we expand $\fS_\lambda$ in the product $ \fS_\alpha \fS_\lambda$:
\begin{equation}\label{expprod}
\fS_\alpha \fS_\lambda  
= \sum_{\sigma \in S_m} (-1)^\sigma \fS_\alpha \HH_{\op \lambda_1+\sigma_1 -1, \lambda_2 + \sigma_2 -2, \dots,\lambda_m + \sigma_m - m\cl},
\end{equation}
where $H_0 = 1$ and $H_i = 0$ if $i <0$.
An iterative use of the Pieri rule [Theorem~\ref{thm:Pieri}] gives
\begin{equation}
\fS_\alpha \HH_{\tau} = \sum_{sh(T)=\gamma/\alpha \atop c(T)=\tau} \fS_\gamma,
\end{equation}
where the sum is over skew immaculate tableaux $T$ and $c(T)$ denotes the content of $T$.
Substituting this back in \eqref{expprod} we get
\begin{equation}\label{expprodTfirst}
\fS_\alpha \fS_\lambda  = \sum_{\sigma \in S_m} \sum_{sh(T)=\gamma/\alpha \atop c(T)=\lambda+\sigma-Id} (-1)^\sigma \fS_\gamma.
\end{equation}
The double  sum on the right is over  pairs $(\sigma,T)$ where  $\sigma\in S_m$ and $T$ is a skew immaculate tableau of shape
$\gamma/\alpha$ and content $c(T) = \lambda+\sigma-Id$. Note that $\sigma$ can be re-constructed from $T$, since $\sigma=c(T)-\lambda+Id$. Let
\begin{equation}
  \label{defnTal}
  \Tal = \left\{T :
      \begin{array}{l}
          \text{$T$ is a skew immaculate tableau} \\
          \text{of inner skew shape $\alpha$ for which} \\
          \text{$c(T)-\lambda+Id$ is a permutation in $S_m$}
      \end{array}
  \right\}
\end{equation}
and for $T$ in $\Tal$, we denote the corresponding permutation by
\begin{equation}
  \label{defnsigma}
  \sigma(T)=c(T)-\lambda+Id.
\end{equation}
Thus, (\ref{expprodTfirst}) can be rewritten as
\begin{equation}\label{expprodT}
\fS_\alpha \fS_\lambda  = \sum_{T \in \Tal} (-1)^{\sigma(T)} \fS_{outsh(T)},
\end{equation}
where $outsh(T)$ is the outer shape of the skew tableau $T$.
 
The theorem will follow from a sign reversing involution on $\Tal$. More specifically,
in Definition \ref{def:Phi} we construct a map $\Phi$ on $\Tal$ with the following properties:
\begin{enumerate}
\item\label{conda} $\Phi(T) \in \Tal$ for $T \in \Tal$.
\item\label{condb} For all $T \in \Tal$, $\Phi^2(T)=T$.
\item\label{condc} If $T \in \Tal$ is Yamanouchi and $\sigma(T)=Id$, then $T$ is a fixed point of $\Phi$.
\item\label{condd} If $T \in \Tal$ is not Yamanouchi or $\sigma(T) \ne Id$, then $T' = \Phi(T) \in \Tal$ is such that:
\begin{enumerate}
\item  The shape of $T$ and $T'$ are equal.
\item $\sigma(T')=t_r \sigma(T)$ for an integer $r$ (where $t_r$ is the transposition which interchanges $r$ and $r+1$).
\end{enumerate}
\end{enumerate}
Properties \eqref{conda}, \eqref{condb}, \eqref{condc}, and \eqref{condd} are proven in Proposition \ref{done!}.
The result follows from the existence of such a map.
\end{proof}

\subsection{Properties and applications of the immaculate Littlewood-Richardson coefficients}
 
For two compositions $\alpha$ and $\nu$, let $\alpha + \nu$ denote the
composition obtained by component-wise addition of the parts of $\alpha$ and
$\nu$:
\begin{gather*}
    \alpha + \nu = \op \alpha_1 + \nu_1, \alpha_2 + \nu_2, \dots, \alpha_m + \nu_m \cl,
\end{gather*}
where $m = \max\{\ell(\alpha), \ell(\nu)\}$ and
where it is understood that
$\alpha_j = 0$ if $j > \ell(\alpha)$
and
$\nu = 0$ if $j > \ell(\nu)$.

Recall that $C_{\alpha, \lambda}^\beta$ denotes the immaculate
Littlewood--Richardson coefficient associated with the compositions $\alpha,
\beta$ and partition $\lambda$, which by Theorem~\ref{ImmLRrule} is the number
of Yamanouchi immaculate tableaux of shape $\beta / \alpha$ and
content~$\lambda$.

\begin{Corollary}\label{cor:newsym}
For compositions $\alpha,\beta,\nu$ and a partition $\lambda$ such that $\ell(\nu)\le\ell(\alpha)$,
  $$C_{\alpha,\lambda}^{\beta}  =  C_{\alpha+\nu,\lambda}^{\beta+\nu}.$$
\end{Corollary}

\begin{proof} This follows by a direct bijection of the two sets counted by $C_{\alpha,\lambda}^{\beta}$ and
$C_{\alpha+\nu,\lambda}^{\beta+\nu}$.
If we have an immaculate tableau $T$ of shape $\beta /\alpha$ and content $\lambda$ in the first set that is
Yamanouchi, we simply map it to the immaculate tableau
of shape $\op \beta+\nu \cl /  \op \alpha +\nu\cl $  (which has the same row distributions, and same first column as $T$)
and content $\lambda$ preserving the rows of $T$. The result is clearly also Yamanouchi since the row word is preserved.
The inverse map is similar.
\end{proof}
 
\begin{Remark}\label{grape}
This implies that  the coefficient $C_{\alpha,\lambda}^\beta$ can be deduced just from knowing
$C_{\op 1^{\ell(\alpha)}\cl\lambda}^{\beta-\alpha+\op 1^{\ell(\alpha)}\cl}$.
\end{Remark}

\begin{Definition}
    For a sequence of integers $\beta=\op\beta_1,\beta_2,\ldots,\beta_m\cl$ we
    define the Schur function indexed by $\beta$ as
    $s_\beta = \det \big[ h_{\beta_i+j-i} \big]_{1\le i,j\le m}$,
    where $h_0 = 1$ and $h_k = 0$ for $k < 0$.
 \end{Definition}
 
\begin{Example}
Let $\lambda = \op 2,1\cl$ and $\alpha = \op 1,1\cl$. Notice that there are the same number of terms in the expression $\fS_{11}\fS_{21}$
as in the expression $\fS_{12}\fS_{21}$ from Example \ref{ex:prod1}.
\[
\fS_{11}\fS_{21} = \fS_{1121} + \fS_{1211} + \fS_{122} + \fS_{131} + \fS_{2111} + \fS_{212} + 2\fS_{221}
+ \fS_{23} + \fS_{311} + \fS_{32}.
\]
We can compute the product of any Schur function with two parts by the Schur function $s_{2,1}$. For instance, $\fS_{32}\fS_{21} = $
\[
\fS_{3221} + \fS_{3311} + \fS_{332} + \fS_{341} + \fS_{4211} +
\fS_{422} + 2\fS_{431} + \fS_{44} + \fS_{521} + \fS_{53}.
\]
Applying the projection to symmetric functions tells us that
\begin{eqnarray*}
s_{32}s_{21} 	& = & s_{3221} + s_{3311} + s_{332} + s_{341} + s_{4211} + s_{422} + 2s_{431} + s_{44} + s_{521} + s_{53} \\
			& = & s_{3221} + s_{3311} + s_{332} + 0 + s_{4211} + s_{422} + 2s_{431} + s_{44} + s_{521} + s_{53},
\end{eqnarray*}
since $s_{3,4,1} = 0$.
\end{Example}
 
We conjecture the following generalization to Corollary~\ref{cor:newsym}.
\begin{Conjecture}[proved by S. X. Li in \cite{SXL}]
    If $\alpha,\beta,\nu$ and $\delta$ are compositions such that
    $\ell(\nu)\le\ell(\alpha)$, then the multiplicative structure constants
    $C_{\alpha, \delta}^{\beta}$ for the immaculate basis
    satisfy
    $$ C_{\alpha,\delta}^{\beta}  =  C_{\alpha+\nu,\delta}^{\beta+\nu}. $$
\end{Conjecture}
 
 Applying the projection map to Theorem  \ref{ImmLRrule} yields the following result about Schur functions.
\begin{Corollary}\label{XImmLRrule}
For two partitions $\mu$ and $\lambda$, \[s_\mu s_\lambda = \sum_{\beta \models |\mu|+|\lambda|}
C_{\mu, \lambda}^\beta s_\beta.\]
\end{Corollary}
 
We have the following relation amongst Schur functions: if $\zeta$ is
a sequence of integers of length $m$ and $\sigma$ a permutation in $S_m$, then 
\begin{equation}\label{eq:2ndrightaction}
    \zeta \star \sigma = \op\zeta_{\sigma_1} - \sigma_1 + 1,
    \zeta_{\sigma_2} - \sigma_2 + 2,
    \ldots, \zeta_{\sigma_m} - \sigma_m + m \cl.
\end{equation}
From Equations \eqref{sweatpotato} and \eqref{skunk} it follows that
$\chi(\fS_\alpha) = s_\alpha = (-1)^\sigma s_{\alpha \star \sigma}$.
The classical Littlewood-Richardson coefficients $c_{\mu, \lambda}^\nu$ in
Theorem~\ref{LRrule} are thus obtained from the immaculate
Littlewood-Richardson coefficients through the following corollary.
 
\begin{Corollary}\label{ImmLRgivesLR}  For partitions $\la, \mu, \nu$ such that $|\la| + |\mu| = |\nu|$,
$$c_{\mu, \lambda}^\nu = \sum_{\sigma\in S_{\ell(\nu)}}
(-1)^{\sigma} C_{\mu, \lambda}^{\nu\star \sigma} $$
with the convention that $C_{\mu, \lambda}^{\nu\star\sigma} =0$
if $\nu\star\sigma$ is not a composition of length $m=\ell(\nu)$ that contains $\mu$.
\end{Corollary}

\begin{Remark}
We have constructed a sign reversing involution to resolve the cancelations in Corollary~\ref{ImmLRgivesLR}
and deduce the classical Littlewood-Richardson rule. We do not include it in this exposition because it is too long and it does not add clarity to the understanding of the rule.

It is however quite interesting that this expression,
combined with Remark \ref{grape}, 
gives an expression for the Littlewood-Richardson coefficients in terms of a
Pieri rule for the immaculate elements since
$$c_{\mu, \lambda}^\nu = \sum_{\sigma\in S_{\ell(\nu)}}
(-1)^{\sigma} C_{[1^{\ell(\mu)}], \lambda}^{\nu\star \sigma - \mu + [1^{\ell(\mu)}]}~.$$
\end{Remark}

\subsection{The details}
As mentioned above, we will define a sign reversing involution $\Phi$ on the set of tableaux $\Tal$
defined in \eqref{defnTal}. This subsection is devoted to constructing $\Phi$ and proving the various
properties of $\Phi$ which were claimed in the proof of Theorem \ref{ImmLRrule}.

In this section, we make use of various left and right actions of permutations
on lists, so for convenience and we collect the definitions of these actions
here.
\begin{itemize}
    \item
        We denote by $\sigma(r)$ the action of a permutation on a single value
        $r$.
    \item
        If $\sigma$ and $\tau$ are two permutations, then we denote by $\sigma
        \circ \tau$ the composition of these permutations:
        $$\left( \sigma \circ \tau \right)(r) = \sigma\left(\tau(r)\right).$$
    \item
        The left action of a permutation $\sigma \in S_n$ on a list of entries
        in $a_i \in \{1, \ldots, n\}$ is
        $$\sigma \cdot \op a_1, a_2, \ldots, a_d\cl = \op\sigma(a_1), \sigma(a_2), \ldots, \sigma(a_d)\cl.$$
    \item
        The right action of a permutation $\sigma \in S_n$ on a list of $n$
        objects is
        $$\op b_1, b_2, \ldots, b_n \cl \cdot \sigma = \op b_{\sigma(1)}, b_{\sigma(2)}, \ldots, b_{\sigma(n)}\cl.$$  This is different than the right action $\star \sigma$ defined in
        Equation \eqref{eq:2ndrightaction}.
\end{itemize}

\begin{Definition}
Given $T \in \Tal$, we construct a tableau $\Ys(T)$, which is a tableau (not
necessarily immaculate) possibly with some empty rows, as follows.
For every cell with entry $r$ in the $i$-th row of $T$
we put a cell with entry $i$ in the $\sigma(T)(r)$-th row of $\Ys(T)$.
We sort the entries of each row of $\Ys(T)$ in increasing order.
\end{Definition}

\begin{Example}\label{gopher}
Let $\lambda=\op2,2,2\cl$ and $\alpha=\op1,2\cl$. Let
$$ T_1= \tikztableausmall{{\boldentry X, 1,1,2},{\boldentry X,\boldentry X,2},{2,3}}.$$
Note that $\sigma(T_1)=c(T_1)-\lambda+Id = \op2,3,1\cl-\op2,2,2\cl+\op1,2,3\cl = \op1,3,2\cl$.
Since $\sigma(T_1)$ is a permutation, $T_1 \in \Tal$. By the construction,
$$\Ys(T_1)=\tikztableausmall{{1,1},{3},{1,2,3}}\,.$$

As a second example, again let $\lambda=\op2,2,2\cl$ and $\alpha=\op1,2\cl$. Let
$$ T_2= \tikztableausmall{{\boldentry X, 1,1},{\boldentry X,\boldentry X,2,2},{1,2}}.$$
Note that $\sigma(T_2)=c(T_2)-\lambda+Id = \op3,3,0\cl-\op2,2,2\cl+\op1,2,3\cl = \op2,3,1\cl$.
Again since $\sigma(T_2)$ is a permutation, $T_2 \in \Tal$. By the construction,
$$\Ys(T_2)=\tikztableausmall{{X},{1,1,3},{2,2,3}}\,,$$
where the shaded empty square indicates that there are no cells in the first row of the diagram.
\end{Example}
 
\begin{Lemma}
\label{Ylemma}
Let $\lambda$ be a partition and $\alpha$ a composition. For $T \in \Tal$, let $\tau = sh(Y(T))$.
\begin{enumerate}
\item\label{lemmaa0}
    $(\tau - Id) = (\lambda - Id) \cdot \sigma(T)^{-1}$.
\item\label{lemmaa} We can recover $T$ uniquely from $\Ys(T)$ (that is, $\Ys$ is injective).
\item\label{lemmab0}
    $\Des(\sigma(T)^{-1}) = \{r : \tau_r < \tau_{r+1} - 1\}$.
\item\label{lemmab1} If $\tau_r < \tau_{r+1}$, then $\tau_r < \tau_{r+1} - 1$.
\item\label{lemmab} $\tau$ is a partition if and only if $\sigma(T)=Id$.
\item\label{lemmac} If $\sigma(T) = Id$, then $\Ys(T)$ is a semi-standard Young tableau
    (of partition shape and strictly increasing down columns) if and only if $\read(T)$ is Yamanouchi.
\end{enumerate}
\end{Lemma}
 
\begin{proof}
    \eqref{lemmaa0}
    By definition $\sigma(T) = c(T) - \lambda + Id$ and so
    \begin{equation}\label{eq:shformula}
   \tau=c(T) \cdot \sigma(T)^{-1}
      = \lambda \cdot \sigma(T)^{-1}+Id-\sigma(T)^{-1}
      = (\lambda-Id) \cdot\sigma(T)^{-1}+Id.
   \end{equation}
 
\eqref{lemmaa} follows from the fact that one can reverse the construction of $\Ys$.
Indeed, the permutation $\sigma(T)$ can be computed from the identity
$(\tau - Id) = (\lambda - Id) \cdot \sigma(T)^{-1}$
since $\tau - Id$ and $\lambda - Id$ have distinct parts.
Then $T$ is the skew tableau of inner skew shape $\alpha$ whose $i$-th row
contains $\sigma(T)^{-1}(u_1)$, $\sigma(T)^{-1}(u_2)$, \dots,
where $u_1$, $u_2$, \dots, are the rows of $Y(T)$ that contain $i$ (listed
according to multiplicity).
 
\eqref{lemmab0}
From \eqref{eq:shformula}, if $d=\sigma(T)^{-1}(r)$ and
$c=\sigma(T)^{-1}(r+1)$, then
\begin{align*}
    \tau_r =\lambda_d-d+r
    \qquad\text{and}\qquad
    \tau_{r+1} =\lambda_c-c+(r+1).
\end{align*}
If $r$ is a descent of $\sigma(T)^{-1}$, then $d > c$ and so
$\lambda_d \leq \lambda_c$. This implies that
\begin{displaymath}
  \tau_r = \lambda_d-d+r <
  \lambda_c-c+r = \tau_{r+1} - 1.
\end{displaymath}
Conversely, if $r$ is an ascent of $\sigma(T)^{-1}$, then $d < c$ and so
$\lambda_d \geq \lambda_c$. This implies that
\begin{displaymath}
  \tau_r = \lambda_d-d+r >
  \lambda_c-c+r = \tau_{r+1} - 1,
\end{displaymath}
and so $\tau_r \geq \tau_{r+1}$.
 
\eqref{lemmab1} If $\tau_r < \tau_{r+1}$, then by the previous sentence, $r$
cannot be an ascent of $\sigma(T)^{-1}$.
Hence, $r$ is a descent of $\sigma(T)^{-1}$, which by \eqref{lemmab0}
implies that $\tau_r < \tau_{r+1} - 1$.
 
\eqref{lemmab} follows from \eqref{lemmab0} by remarking that the
identity permutation is the only permutation with no descents.
 
To prove \eqref{lemmac}, note that in $\Ys(T)$, the $(i,j)$ entry is a $k$ if in $T$, the $j$-th appearance of the letter
$i$ in $\read(T)$ comes from a cell in row $k$. The result follows.
\end{proof}
  
We now define a mapping on $Y(\Tal)$, which we will prove
is an involution.
We say that a cell $x$ not in the first row of a tableau and containing the
value $a$ is \emph{nefarious} if the cell above $x$ is either empty or it
contains $b$ with $b \geq a$:
 $$ \tikztableausmall{{ X},{$a$}} \qquad \text{ or }\qquad \tikztableausmall{{$b$},{$a$}}.$$
The \emph{most nefarious cell} of a tableau is the bottom-most nefarious cell
contained in the left-most column of the tableau that contains a nefarious cell.
 
\begin{Definition}\label{def:Theta}
For a partition $\lambda$ and a composition $\alpha$, we define a map $\Theta$ on $\Ys(\Tal)$ as follows
(\textit{cf.} Figure \ref{fig:Theta2}).
Let $Y(T) \in Y(\Tal)$.
 
\begin{enumerate}
\item\label{Theta1}
If $Y(T)$ contains no nefarious cells
(equivalently, if $Y(T)$ is semistandard
and of partition shape),
then $Y(T)$ is a fixed point of $\Theta$.
 
\item\label{Theta2}
Otherwise, let $x$ be the most nefarious cell of $Y(T)$.
 
\begin{enumerate}
  \item\label{Theta2a}
      If the cell $y$ above $x$ is not empty, then define $\Theta(Y(T))$ to be the tableau obtained from $Y(T)$ by moving:
      \begin{itemize}
          \item
              all the cells strictly to the right of $x$ into the row above $x$; and
          \item
              all the cells weakly to the right of $y$ (including $y$) into the row containing $x$.
      \end{itemize}
 
  \item\label{Theta2b}
  Otherwise, define $\Theta(Y(T))$ to be the tableau obtained from $Y(T)$ by moving all the cells strictly to the right
  of $x$ into the row above $x$.
\end{enumerate}
\end{enumerate}
\end{Definition}
 
\begin{figure}[h!]
\begin{center}
\begin{tikzpicture}[scale=0.5, every node/.style={anchor=south}]
  \coordinate (o) at (0, 0);
  \node[anchor=south west] at ($(o) +(-6,0)$) {\small row $r+1$};
  \node[anchor=south west] at ($(o) +(-6,1)$) {\small row $r$};
  \draw[color=black] (o) rectangle +(-2, 1);
  \node at ($(o) +(-1,0)$) {$\cdots$};
  \draw[color=black] ($(o) +(0,1)$) rectangle +(-2, 1);
  \node at ($(o) +(-1,1)$) {$\cdots$};
  \draw[color=black] (o) rectangle +(1, 1);
  \node at ($(o) +(0.5,0)$) {$x$};
  \draw[color=black] ($(o) + (0, 1)$) rectangle +(1, 1);
  \node[color=RoyalBlue] at ($(o) +(0.5,1)$) {$y$};
  \draw[color=black] ($(o) + (1, 0)$) rectangle +(8, 1);
  \node[color=ForestGreen] at ($(o) +(5,0)$) {$u$};
  \draw[color=black] ($(o) + (1, 1)$) rectangle +(5, 1);
  \node[color=BrickRed] at ($(o) +(3.5,1)$) {$v$};
  \node at ($(o) +(11,0.5)$) {${\buildrel \Theta\over\longmapsto}$};
  \coordinate (o) at (15, 0);
  \draw[color=black] (o) rectangle +(-2, 1);
  \node at ($(o) +(-1,0)$) {$\cdots$};
  \draw[color=black] ($(o) +(0,1)$) rectangle +(-2, 1);
  \node at ($(o) +(-1,1)$) {$\cdots$};
  \draw[color=black] (o) rectangle +(1, 1);
  \node at ($(o) +(0.5,0)$) {$x$};
  \draw[color=black] ($(o) + (1, 0)$) rectangle +(1, 1);
  \node[color=RoyalBlue] at ($(o) +(1.5,0)$) {$y$};
  \draw[color=black] ($(o) + (0, 1)$) rectangle +(8, 1);
  \node[color=ForestGreen] at ($(o) +(4,1)$) {$u$};
  \draw[color=black] ($(o) + (2, 0)$) rectangle +(5, 1);
  \node[color=BrickRed] at ($(o) +(4.5,0)$) {$v$};
\end{tikzpicture}
 
\medskip
\medskip
 
\begin{tikzpicture}[scale=0.5, every node/.style={anchor=south}]
  \coordinate (o) at (0, 0);
  \node[anchor=south west] at ($(o) +(-6,0)$) {\small row $r+1$};
  \node[anchor=south west] at ($(o) +(-6,1)$) {\small row $r$};
  \draw[color=black] (o) rectangle +(-2, 1);
  \node at ($(o) +(-1,0)$) {$\cdots$};
  \draw[color=black] ($(o) +(0,1)$) rectangle +(-2, 1);
  \node at ($(o) +(-1,1)$) {$\cdots$};
  \draw[color=black] (o) rectangle +(1, 1);
  \node at ($(o) +(0.5,0)$) {$x$};
  \draw[color=black] ($(o) + (1, 0)$) rectangle +(8, 1);
  \node[color=ForestGreen] at ($(o) +(5,0)$) {$u$};
  \node at ($(o) +(11,0.5)$) {${\buildrel \Theta\over\longmapsto}$};
  \coordinate (o) at (15, 0);
  \draw[color=black] (o) rectangle +(-2, 1);
  \node at ($(o) +(-1,0)$) {$\cdots$};
  \draw[color=black] ($(o) +(0,1)$) rectangle +(-2, 1);
  \node at ($(o) +(-1,1)$) {$\cdots$};
  \draw[color=black] (o) rectangle +(1, 1);
  \node at ($(o) +(0.5,0)$) {$x$};
  \draw[color=black] ($(o) + (0, 1)$) rectangle +(8, 1);
  \node[color=ForestGreen] at ($(o) +(4,1)$) {$u$};
\end{tikzpicture}
\end{center}
\caption{The effect of $\Theta$ on the cells to the right
  of the two types of a nefarious cell $x$.}
\label{fig:Theta2}
\end{figure}
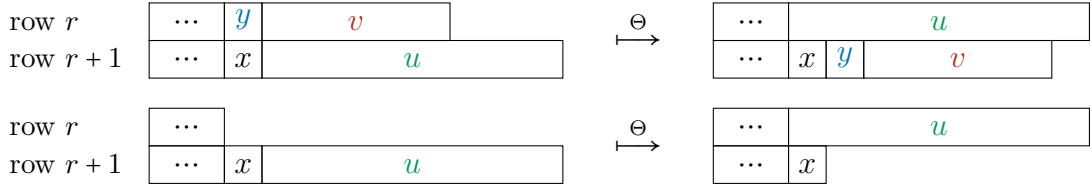
 
\begin{Example}\label{weasel}
Continuing Example \ref{gopher}, the most nefarious cell of $\Ys(T_1)$ is in the third row of the first column. Applying $\Theta$ yields:
$$\Theta\left(\tikztableausmall{{1,1},{\textcolor{RoyalBlue}3},{\textbf1,\textcolor{ForestGreen}2,\textcolor{ForestGreen}3}}\right) = \tikztableausmall{{1,1},{\textcolor{ForestGreen}2,\textcolor{ForestGreen}3},{\textbf1,\textcolor{RoyalBlue}3}}\,.$$
In $\Ys(T_2)$, the most nefarious cell is in
the second row of the first column and so
$$\Theta\left(  \tikztableausmall{{X},{\textbf1,\textcolor{ForestGreen}1,\textcolor{ForestGreen}3},{2,2,3}}\right) =   \tikztableausmall{{\textcolor{ForestGreen}1,\textcolor{ForestGreen}3},{\textbf1},{2,2,3}}\,.$$
\end{Example}
 
\begin{Lemma}
  If $Y(T) \in Y(\Tal)$ contains a nefarious cell, then
  $\Theta(Y(T)) \neq Y(T)$.
\end{Lemma}
\begin{proof}
  Suppose the most nefarious cell $x$ of $Y(T)$ lies in row $r+1$.
  Let $\tau = sh(Y(T))$.
 
  If $\tau_r \geq \tau_{r+1}$, then the cell $y$ above $x$ is not empty
  and so we find ourselves in case \eqref{Theta2a} of Definition \ref{def:Theta}.
  Hence, the number of cells that move from row $r$ into row $r+1$
  is greater than the number of cells that move from row $r+1$ into row $r$.
  In particular, $\Theta(Y(T))\ne Y(T)$.
 
  Suppose instead that $\tau_r < \tau_{r+1}$. Then we have
  that $\tau_r < \tau_{r+1} - 1$ by \ref{Ylemma}\eqref{lemmab1}.
  Hence, the number of cells that move from row $r+1$ into row $r$
  is greater than the number of cells that move from row $r$ into row $r+1$.
  This insures that $\Theta(Y(T))\ne Y(T)$ in this case as well.
\end{proof}
 
\begin{Lemma}\label{permutationforT'} For a composition $\alpha$, a partition $\lambda$ and $T\in \Tal$,
if there exists a $T' \in \Tal$ for which $Y(T') = \Theta(Y(T))$ and $T' \neq T$, then
there exists an $1 \leq r < \ell(outsh(T))$ such that $\sigma(T') = t_r \circ \sigma(T)$.
\end{Lemma}
 
\begin{proof}
   Suppose that $\Theta(Y(T)) = Y(T')$ and that $T' \neq T$. Then $Y(T)$ is
   not a fixed point of $\Theta$ and so it contains a nefarious cell. Suppose
   the most nefarious cell $x$ of $Y(T)$ lies in row $r+1$.
   Then the shapes of $\Theta(Y(T))$ and $Y(T)$ differ only in rows $r$ and
   $r+1$. More explicitly,
   \begin{gather*}
       outsh(\Theta(\Ys(T))) = \big(outsh({\Ys(T)}) + t_r - Id\big) \cdot t_r
   \end{gather*}
   since $t_r - Id=(0,\ldots,0,1,- 1,0,\ldots 0)$ is the vector whose nonzero
   entries occur in positions $r$ and $r+1$.
   Using Equation \eqref{eq:shformula}, we rewrite $outsh(Y(T))$ to obtain
   \begin{align*}
       outsh(\Theta(\Ys(T)))
       &= \big(\lambda \cdot \sigma(T)^{-1} + t_r -\sigma(T)^{-1}\big) \cdot t_r  \\
       &= \lambda \cdot \left(\sigma(T)^{-1} \circ t_r\right)
          + Id - \sigma(T)^{-1} \circ t_r.
   \end{align*}
   But, also from Equation \eqref{eq:shformula}, we have
       $outsh(Y(T')) =\big(\lambda - Id\big) \cdot \sigma(T')^{-1} + Id$
   so that
   \begin{align*}
       \big(\lambda - Id\big) \cdot \sigma(T')^{-1}
       & = outsh(Y(T')) - Id \\
       & = \left(
               \lambda \cdot \left(\sigma(T)^{-1} \circ t_r\right)
               + Id - \sigma(T)^{-1} \circ t_r
           \right)
       - Id \\
       & = \big( \lambda - Id \big) \cdot \left(\sigma(T)^{-1} \circ t_r\right)
   \end{align*}
   Since all the parts of $\lambda - Id$ are distinct, we conclude that
   $\sigma(T')^{-1} = \sigma(T)^{-1} \circ t_r$.
\end{proof}
 
\begin{Lemma} \label{lemma:Yinversexists}
For every immaculate tableau $T \in \Tal$, there exists an immaculate tableau
$T' \in \Tal$ such that $Y( T' ) = \Theta(\Ys(T))$.
\end{Lemma}
\begin{proof}
The proof of Lemma \ref{Ylemma}\eqref{lemmaa} gives a procedure to compute the
preimage of $\Theta(Y(T))$ under the map $Y$.
We will show that the resulting skew tableau $T'$ belongs to $\Tal$.
To simplify the notation, let $\sigma = \sigma(T)$.
 
Suppose that the most nefarious cell $x$ of $Y(T)$ occurs in row $r+1$.
Let $T'$ denote the skew tableau of inner skew shape $\alpha$
whose $j$-th row contains
$(t_r \circ \sigma)^{-1}(v_1)$,
\dots,
$(t_r \circ \sigma)^{-1}(v_a)$
(listed in weakly increasing order),
where $v_1, \dots, v_a$ are the rows of $\Theta(Y(T))$ that contain $j$,
with each row listed as often as $j$ appears in it.
For examples of this, see Examples \ref{duckfish} and \ref{goat}.
By construction, the entries in each row of $T'$ are weakly increasing
and $T'$ has the same shape as $T$, so it
remains to show that the entries in the first column of $T'$ are strictly
increasing. We will show that the first column of $T'$ is equal to the first
column of $T$.

We begin by making a few observations.
\begin{description}
    \item[Claim 1]
        \emph{The tableau $T'$ is obtained from $T$ by changing some
            occurrences of $\sigma^{-1}(r)$ to $\sigma^{-1}(r+1)$ and some
            occurrences of $\sigma^{-1}(r+1)$ to $\sigma^{-1}(r)$.}

        \smallskip
        Let $e_1, \dots, e_l$ be the entries in the $i$-th row of $T$.
        Then $i$ appears in rows
            $\sigma(e_1), \dots, \sigma(e_l)$
        of $Y(T)$, with each row listed as often as $i$ appears in it.

        Since $\Theta$ moves cells from $r$ to row $r+1$ and cells from row
        $r+1$ to row $r$,
        the list of rows of $\Theta(Y(T))$ that contain $i$
        is obtained from the list
        $\sigma(e_1), \dots, \sigma(e_l)$
        by changing some occurrences of $r$ to $r + 1$
        and by changing some occurrences of $r + 1$ to $r$.
        Denote this transformed list by
        $\theta_1, \dots, \theta_l$,
        where
        \begin{align*}
            \theta_k &= \sigma(e_k)
                \text{~if~} \sigma(e_k) \notin \left\{ r, r+1 \right\}
                \text{; and} \\
            \theta_k &\in \left\{ r, r + 1 \right\}
                \text{~if~} \sigma(e_k) \in \left\{ r, r+1 \right\}.
        \end{align*}
        Note that $\theta_k \in \left\{r, r+1\right\}$
        iff $\sigma(e_k) \in \left\{r, r+1\right\}$.

        Row $i$ of $T'$ is obtained by applying $(t_r \circ \sigma)^{-1}$
        to the list $\theta_1, \dots, \theta_l$.
        If $e_k$ is neither $\sigma^{-1}(r)$ nor $\sigma^{-1}(r+1)$,
        then $t_r$ fixes $\theta_r$ since $\theta_k \notin \left\{ r, r + 1 \right\}$;
        hence $(t_r \circ \sigma)^{-1}(\theta_k) = \sigma^{-1}(\theta_k) = e_k$.
        Consequently, the entry $e_k$ is unchanged if $e_k$
        is neither $\sigma^{-1}(r)$ nor $\sigma^{-1}(r+1)$.
        On the other hand,
        if $e_k \in \left\{ \sigma^{-1}(r), \sigma^{-1}(r + 1) \right\}$,
        then $t_r(\theta_k) \in \left\{r,r+1\right\}$ since
        $\theta_k \in \left\{r, r+1\right\}$;
        hence $(t_r \circ \sigma)^{-1}(\theta_k)
        \in \left\{ \sigma^{-1}(r), \sigma^{-1}(r + 1) \right\}$.

        \medskip

    \item[Claim 2]
        \emph{If $j > \ell(\alpha)$, then the first entry in the $j$-th row of $T$ is}
        \begin{gather*}
           f_j(T) = \min\left\{ \sigma^{-1}(s) :
                     \text{row $s$ of $Y(T)$ contains $j$} \right\};
        \end{gather*}
        \emph{and the first entry in the $j$-th row of $T'$ is}
        \begin{gather*}
           f_j(T') = \min\left\{ \left(t_r \circ \sigma\right)^{-1}(s) :
                     \text{row $s$ of $\Theta(Y(T))$ contains $j$} \right\}.
        \end{gather*}

        \medskip

    \item[Claim 3]
        \emph{The maximal entry in row $\sigma(f_j(T))$ of $Y(T)$ is $j$.}

        \smallskip
        Since $T$ is a skew immaculate tableau and $f_j(T)$ is the first entry in row
        $j$ of $T$, it follows that $f_j(T)$ does not appear in rows $j+1, j+2, \dots$
        of $T$.
        Hence, none of $j+1, j+2, \dots$ appear in row $\sigma(f_j(T))$ of $Y(T)$.
\end{description}

\medskip

Now suppose that the first column of $T$ is not equal to the first column of $T'$
and let $j > \ell(\alpha)$ be a row in which the first entries of $T$ and $T'$
differ: $f_j(T) \neq f_j(T')$. By Claim 1, these entries are necessarily
$\sigma^{-1}(r)$ and $\sigma^{-1}(r+1)$.
 
Below we will make reference to the following diagram
illustrating the effect of $\Theta$ on the rows $r$ and $r+1$ of $Y(T)$;
note that it is possible that the cells containing $c$, $d$ and $b$ do not
exist.
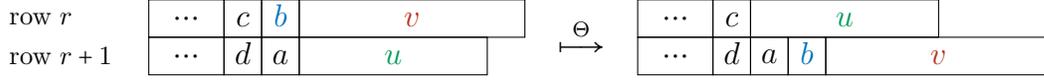
\begin{center}
\begin{tikzpicture}[scale=0.5]
   \coordinate (o) at (0, 0);
   \node[anchor=west] at ($(o) +(-7,0.5)$) {\footnotesize row $r+1$};
   \node[anchor=west] at ($(o) +(-7,1.5)$) {\footnotesize row $r$};
   \draw[color=black] ($(o) +(-1,0)$) rectangle +(-2, 1);
   \node[anchor=south] at ($(o) +(-2,0)$) {$\cdots$};
   \draw[color=black] ($(o) +(-1,1)$) rectangle +(-2, 1);
   \node[anchor=south] at ($(o) +(-2,1)$) {$\cdots$};
   \draw[color=black] ($(o) + (0,1)$) rectangle +(-1, 1);
   \node[anchor=south] at ($(o) +(-0.5,1)$) {$c$};
   \draw[color=black] (o) rectangle +(-1, 1);
   \node[anchor=south] at ($(o) +(-0.5,0)$) {$d$};
   \draw[color=black] (o) rectangle +(1, 1);
   \node[anchor=south] at ($(o) +(0.5,0)$) {$a$};
   \draw[color=black] ($(o) + (0, 1)$) rectangle +(1, 1);
   \node[color=RoyalBlue,anchor=south] at ($(o) +(0.5,1)$) {$b$};
   \draw[color=black] ($(o) + (1, 0)$) rectangle +(5, 1);
   \node[color=ForestGreen,anchor=south] at ($(o) +(3.5,0)$) {$u$};
   \draw[color=black] ($(o) + (1, 1)$) rectangle +(6, 1);
   \node[color=BrickRed,anchor=south] at ($(o) +(4,1)$) {$v$};
   \node at ($(o) +(8.5,1)$) {${\buildrel \Theta\over\longmapsto}$};
   \coordinate (o) at (13, 0);
   \draw[color=black] ($(o) +(-1,0)$) rectangle +(-2, 1);
   \node[anchor=south] at ($(o) +(-2,0)$) {$\cdots$};
   \draw[color=black] ($(o) +(-1,1)$) rectangle +(-2, 1);
   \node[anchor=south] at ($(o) +(-2,1)$) {$\cdots$};
   \draw[color=black] ($(o) + (0,1)$) rectangle +(-1, 1);
   \node[anchor=south] at ($(o) +(-0.5,1)$) {$c$};
   \draw[color=black] (o) rectangle +(-1, 1);
   \node[anchor=south] at ($(o) +(-0.5,0)$) {$d$};
   \draw[color=black] (o) rectangle +(1, 1);
   \node at ($(o) +(0.5,0.5)$) {$a$};
   \draw[color=black] ($(o) + (1, 0)$) rectangle +(1, 1);
   \node[color=RoyalBlue,anchor=south] at ($(o) +(1.5,0)$) {$b$};
   \draw[color=black] ($(o) + (0, 1)$) rectangle +(5, 1);
   \node[color=ForestGreen,anchor=south] at ($(o) +(2.5,1)$) {$u$};
   \draw[color=black] ($(o) + (2, 0)$) rectangle +(6, 1);
   \node[color=BrickRed,anchor=south] at ($(o) +(5,0)$) {$v$};
\end{tikzpicture}
\end{center}

\textbf{Case 1.}
\emph{Suppose that $f_j(T) = \sigma^{-1}(r)$
      and $f_j(T') = \sigma^{-1}(r+1) = (t_r \circ \sigma)^{-1}(r)$.}
We derive a contradiction by showing
    $\sigma^{-1}(r) > \sigma^{-1}(r+1)$
and
    $\sigma^{-1}(r) < \sigma^{-1}(r+1)$.

\medskip
\emph{Case 1(a). Suppose that the cell directly above $a$ belongs to $Y(T)$.}

\smallskip

We first argue that $\sigma^{-1}(r) > \sigma^{-1}(r+1)$.
The equality $f_j(T) = \sigma^{-1}(r)$ implies that $j$ is the maximal entry
in row $r$ of $Y(T)$, by Claim 3, which means that $j$ is contained in row $r+1$ of
$\Theta(Y(T))$.
Since $f_j(T')$ is such that
    $f_j(T') \leq (t_r \circ \sigma)^{-1}(s)$
for all rows $s \neq r$ of $\Theta(Y(T))$ that contain $j$,
it follows that
\begin{gather*}
    \sigma^{-1}(r) = (t_r \circ \sigma)^{-1}(r+1)
    \geq f_j(T') = (t_r \circ \sigma)^{-1}(r) = \sigma^{-1}(r+1).
\end{gather*}
It follows that $\sigma^{-1}(r) > \sigma^{-1}(r+1)$ since $\sigma$ is
a permutation.

We now argue that $\sigma^{-1}(r) < \sigma^{-1}(r+1)$.
The equality
$f_j(T') = (t_r \circ \sigma)^{-1}(r)$
implies that row $r$ of $\Theta(Y(T))$ contains $j$.
We argue that $j$ occurs in $u$.
If $c$ does not belong to $Y(T)$, then the $d$ is a nefarious cell,
contradicting that $a$ is the most nefarious cell.
Hence, $c$ belongs to $Y(T)$.
If $j \leq c$, then $j \leq c < d \leq a \leq
b$ since the cell containing $a$ is the most nefarious cell of $Y(T)$. This
inequality contradicts the fact that $j$ is the maximal entry in row $r$ of
$Y(T)$. Hence, $j$ occurs in $u$, which means that $j$ also occurs in row $r+1$
of $Y(T)$. This implies that $\sigma^{-1}(r) = f_j(T) \leq \sigma^{-1}(r+1)$,
which in turn implies that $\sigma^{-1}(r) < \sigma^{-1}(r+1)$ since $\sigma$
is a permutation.

\medskip
\emph{Case 1(b). Suppose the cell directly above $a$ does not belong to $Y(T)$.}

\smallskip

To prove that $\sigma^{-1}(r) > \sigma^{-1}(r+1)$, it suffices to show that
$r$ is a descent of the permutation $\sigma^{-1}$.
Since the cell directly above $a$ does not belong to the shape of $Y(T)$,
there are more cells in row $r+1$ than in row $r$ of $Y(T)$.
If we denote $sh(Y(T))$ by $\tau$, then we have $\tau_r < \tau_{r+1}$.
By Lemma~\ref{Ylemma}(\ref{lemmab1}),
it follows that $\tau_r < \tau_{r + 1} - 1$,
and so $r \in \Des(\sigma^{-1})$ (by Lemma~\ref{Ylemma}(\ref{lemmab0})).

We now argue that $\sigma^{-1}(r) < \sigma^{-1}(r+1)$.
By Claim 3, the equality $f_j(T) = \sigma^{-1}(r)$ implies that $j$ is the
maximal entry in row $r$ of $Y(T)$.
Since the cell directly above $a$ does not belong to $Y(T)$, we have that
$j$ occurs immediately above an entry $e$ that is strictly to the left of $a$.
Since the cell containing $a$ is the most nefarious cell of $Y(T)$,
it follows that $j < e$. Hence, $j < a$ since $e \leq a$ because the rows of
$Y(T)$ are sorted.

Since $a > j > \ell(\alpha)$, Claim 2 tells us that
the first entry in the $a$-th row of $T$ is $f_a(T)$;
and the first entry in the $j$-th row of $T$ is $f_j(T)$.
Since $T$ is immaculate, the entries in the first column
are strictly increasing from top to bottom, and so $f_j(T) < f_a(T)$.
Hence, we have that
$$\sigma^{-1}(r) = f_j(T) < f_a(T)
= \min\left\{ \sigma^{-1}(s) : \text{$a$ appears in row $s$ of $Y(T)$} \right\}
\leq \sigma^{-1}(r+1),$$
where the last inequality follows from the fact that $a$ appears in row $r+1$
of $Y(T)$.

\medskip
\textbf{Case 2.}
\emph{Suppose that $f_j(T) = \sigma^{-1}(r+1)$
      and $f_j(T') = \sigma^{-1}(r) = (t_r \circ \sigma)^{-1}(r+1)$.}
We derive a contradiction by showing
    $\sigma^{-1}(r) > \sigma^{-1}(r+1)$
and
    $\sigma^{-1}(r) < \sigma^{-1}(r+1)$.

\medskip

\emph{Proof that $\sigma^{-1}(r) > \sigma^{-1}(r+1)$.}
If the cell labelled $b$ does not belong to $\tau = sh(Y(T))$,
then $\tau_r < \tau_{r+1}$,
and so $\tau_r < \tau_{r+1} - 1$ by Lemma~\ref{Ylemma}(\ref{lemmab1})
and $r$ is a descent of $\sigma^{-1}$ by Lemma~\ref{Ylemma}(\ref{lemmab0}).
So suppose that the cell labelled $b$ belongs to $Y(T)$.

Since $f_j(T') = (t_r \circ \sigma)^{-1}(r+1)$,
we have that $j$ occurs in row $r+1$ of $\Theta(Y(T))$.
If $j \geq b$, then $b = j$ or $j$ occurs in $v$.
In both cases, $j$ appears in row $r$ of $Y(T)$.
Hence, $\sigma^{-1}(r) \geq f_j(T) = \sigma^{-1}(r+1)$,
and so $\sigma^{-1}(r) > \sigma^{-1}(r+1)$ since $\sigma^{-1}$ is
a permutation.

If $j < b$, then $j$ appears to the left of $b$ in row $r + 1$ of
$\Theta(Y(T))$. Hence, $j \leq a$, since $a$ is immediately to the left of $b$.
And since $j$ is the maximal entry in row $r+1$ of $Y(T)$,
we have that $a = j$ and all the entries of $u$ are equal to $j$.
Now consider the tableau $T$. Since $T$ is immaculate and $b > j$,
the first entry $e_b$ in the $b$-th row of $T$ is greater
than the first entry $e_j$ in the $j$-th row of $T$.
Since $Y(T)$ has a cell with entry $b$ in its $r$-th row,
it follows the $\sigma^{-1}(r)$ appears in row $b$ of $T$,
and so $e_b \leq \sigma^{-1}(r)$.
By Claim 2, $e_j = f_j(T) = \sigma^{-1}(r+1)$, and so
$\sigma^{-1}(r) \geq e_b > e_j = \sigma^{-1}(r+1)$.

\medskip

\emph{Proof that $\sigma^{-1}(r) < \sigma^{-1}(r+1)$.}
We deal with the two cases corresponding to whether the block of cells labelled
$u$ in the diagram is empty.

\emph{Suppose at least one cell directly to the right of
    $a$ belongs to $Y(T)$.}
Since $f_j(T) = \sigma^{-1}(r + 1)$, Claim 3 implies that the maximal entry in
row $r + 1 = \sigma(f_j(T))$ of $Y(T)$ is $j$.
Hence, there is a cell to the right of $a$ that contains a $j$.
After applying $\Theta$, this cell moves into row $r$.
Hence, row $r$ of $\Theta(Y(T))$ contains $j$ and so
\begin{gather*}
    \sigma^{-1}(r) = f_j(T')
    \leq (t_r \circ \sigma)^{-1}(r) = \sigma^{-1}(r+1),
\end{gather*}
which in turn implies that $\sigma^{-1}(r) < \sigma^{-1}(r+1)$,
since $\sigma$ is a permutation.

\emph{Suppose the cells directly to the right of $a$ do not belong
    to $Y(T)$.}
Let $\tau = sh(Y(T))$. We will show that $\tau_r \geq \tau_{r+1}$.
Suppose that $\tau_r < \tau_{r+1}$. Then $\tau_r < \tau_{r+1} - 1$
by Lemma~\ref{Ylemma}(\ref{lemmab1}).
Since the cell labelled $a$ is the last cell of row $r+1$,
this implies that the cells labelled $b$ and $c$ do not belong to $Y(T)$.
But then the cell labelled $d$ is a nefarious cell of $Y(T)$,
contradicting the fact that the cell labelled $a$ is the most
nefarious cell of $Y(T)$.
Hence, $\tau_r \geq \tau_{r + 1}$.
By Lemma~\ref{Ylemma}(\ref{lemmab0}), $r$ is not a descent
of $\sigma^{-1}$ and so $\sigma^{-1}(r) < \sigma^{-1}(r+1)$.
\end{proof}
 
\begin{Example}\label{duckfish}
Continuing Example \ref{gopher} and Example \ref{weasel} we see that there exists a tableau $T_1' \in \Tal$ for which $Y(T_1') = \Theta(Y(T_1))$:
$$T_1=
\tikztableausmall{{\boldentry X, 1,1,2},{\boldentry X,\boldentry X,2},{2,3}} {\buildrel \Ys\over\longmapsto}
\tikztableausmall{{1,1},{3},{1,2,3}} {\buildrel \Theta\over\longmapsto}
\tikztableausmall{{1,1},{2,3},{1,3}} {\buildrel Y\over\longmapsfrom}
\tikztableausmall{{\boldentry X, 1,1,{\bf 3}},{\boldentry X,\boldentry X,2},{2,3}} =T'_1.
$$
Furthermore, $\sigma(T_1') = \op2,2,2\cl - \op2,2,2\cl + \op1,2,3\cl = \op1, 2, 3\cl = t_2 \circ \op1,3,2\cl = t_2 \circ \sigma(T_1)$.
 
Similarly, there exists a tableau $T_2' \in \Tal$ for which $Y(T_2') = \Theta(Y(T_2))$:
 
$$T_2=
\tikztableausmall{{\boldentry X, 1,1},{\boldentry X,\boldentry X,2,2},{1,2}}  {\buildrel \Ys\over\longmapsto}
\tikztableausmall{{X},{1,1,3},{2,2,3}} {\buildrel \Theta\over\longmapsto}
\tikztableausmall{{1,3},{1},{2,2,3}} {\buildrel Y\over\longmapsfrom}
\tikztableausmall{{\boldentry X, 1,{\bf 3}},{\boldentry X,\boldentry X,2,2},{1,2}} =T'_2~.
$$
Furthermore, $\sigma(T_2') = \op2,3,1\cl - \op2,2,2\cl + \op1,2,3\cl = \op1, 3, 2 \cl = t_1 \circ \op2,3,1\cl = t_1 \circ \sigma(T_2)$.
\end{Example}
 
\begin{Example}\label{goat}
Let $\alpha = \op 1,2\cl$, $\lambda = \op3,3,0\cl$ and let $T$ be the tableau:
$$\tikztableausmall{{\boldentry X, 2},{\boldentry X,\boldentry X,1,1},{1,1},{2}}.$$
Then $\sigma(T) = c(T) - \lambda  + Id = \op4,2,0\cl - \op3,3,0\cl + \op1,2,3\cl = \op2,1,3\cl$, so $T \in \Tal$.
$$
\tikztableausmall{{\boldentry X, 2},{\boldentry X,\boldentry X,1,1},{1,1},{2}}  {\buildrel \Ys\over\longmapsto}
\tikztableausmall{{1,4},{2,2,3,3}} {\buildrel \Theta\over\longmapsto}
\tikztableausmall{{1,3,3},{2,2,4}} {\buildrel Y^{-1}\over\longmapsto}
\tikztableausmall{{\boldentry X, {\bf 1} },{\boldentry X,\boldentry X,{\bf 2},{\bf 2}},{1,1},{2}}$$
\end{Example}

\begin{Proposition}\label{involution}
  $\Theta$ is an involution on $Y(\Tal)$.
\end{Proposition}
\begin{proof}
  Suppose $Y(T) \in Y(\Tal)$.
  By construction, if $x$ is the most nefarious cell of $Y(T)$, then $x$ is
  also a nefarious cell of $\Theta(Y(T))$.
  Moreover, it is the most nefarious cell of $\Theta(Y(T))$ since the only
  cells changed by $\Theta$ occur to the right of and above $x$. By Lemma \ref{lemma:Yinversexists}, $\Theta(Y(T)) \in Y(\Tal)$.
  Thus, part \eqref{Theta2} of the definition of $\Theta$ applies to the same
  nefarious cell $x$ and it will undo the modifications effected by $\Theta$.
\end{proof}

\begin{Definition}\label{def:Phi}
For a composition $\alpha$ and a partition $\lambda$,
we define the map $\Phi$ on $\Tal$ by $\Phi(T) = Y^{-1}(\Theta ( Y(T)))$.
\end{Definition}

\begin{Proposition}\label{done!} For a composition $\alpha$ and a partition $\lambda$,
\begin{enumerate}
\item\label{conda2} $\Phi(T) \in \Tal$ for $T \in \Tal$.
\item\label{condb2} For all $T \in \Tal$, $\Phi^2(T)=T$.
\item\label{condc2} If $T \in \Tal$ is Yamanouchi and $\sigma(T)=Id$, then $T$ is a fixed point of $\Phi$.
\item\label{condd2} If $T \in \Tal$ is not Yamanouchi or $\sigma(T) \ne Id$, then $T' = \Phi(T) \in \Tal$ is such that:
\begin{enumerate}
\item\label{Aa}  The shape of $T$ and $T'$ are equal.
\item\label{Bb} $\sigma(T')=t_r \sigma(T)$ for an integer $r$ (where $t_r$ is the transposition which interchanges $r$ and $r+1$).
\end{enumerate}
\end{enumerate}
\begin{proof}
Condition \eqref{conda2} is a direct consequence of Lemma \ref{lemma:Yinversexists}.
Condition \eqref{condb2} is a direct consequence of Proposition \ref{involution}.
Condition \eqref{condc2} follows from the definition of $\Theta$ and Lemma \ref{Ylemma} \eqref{lemmac}.
Condition \eqref{Aa} follows from the fact that $\Theta$ preserves contents and $Y$
interchanges shape with content and Condition \eqref{Bb} follows from Lemma \ref{permutationforT'}.
\end{proof}
\end{Proposition}
 
\begin{moredetails}
The involution $\Theta$ above can be modified and used for the general multiplication of immaculate basis $\fS_\alpha \fS_\delta$.
Using the same reasoning as above, we have
\begin{equation}\label{expGprodT}
\fS_\alpha \fS_\delta  = \sum_{\sigma \in S_m} \sum_{sh(T)=\gamma/\alpha \atop c(T)={\delta+\sigma-Id}} (-1)^\sigma \fS_\gamma.
\end{equation}
The double  sum on the right is over  pairs $(\sigma,T)$ where  $\sigma\in S_m$ and $T$ is an immaculate tableau of shape
$\gamma/\alpha$ and content $\delta+\sigma-Id$ (componentwise).
We apply sign reversing involution among the pair $(\sigma,T)$ as before using $Y_{\sigma}(T)$.
Here we have to remember the dependence of $Y$ on $\sigma$ to characterize the image and understand the inverse.
In the case when $\delta$ is not a partition we could get pairs
$$   \tikztableausmall{{ $b'$,X},{$a'$,$a$}}  \text{ or }  \tikztableausmall{{ $b$},{$a$}} $$
where $b'$ or $a$ correspond to an element in the first column of $T$. If we where to apply the involution in those cases,
we could get a tableau that is not in the
image of $Y_{\sigma'}$ (hence not in the sum above). So we have to add those as fixed point in our definitions.
This create additional fix points as well as fixed points
for permutation $\sigma\ne Id$. But this is to be expected as the general rule contains negative numbers.
We begin by modifying the definition of Yamanouchi:
\begin{Definition}
A pair $(\sigma,T)$ of a permutation and a skew immaculate tableau $T$ will be called I-Yamanouchi if for every rows $r,r+1$ of
$Y_\sigma(T)$ the left-most pair (if it is present)
$$   \tikztableausmall{{ $b'$,X},{$a'$,$a$}}  \text{ or }  \tikztableausmall{{ $b$},{$a$}} $$
is such that $\theta_r(Y_\sigma(T))$ is not in the image of $Y_{(r\,\,r+1)\circ \sigma}$.
\end{Definition}

This definition is rather hard to use and it is difficult to give a direct characterization of I-Yamanouchi tableaux.
Given this definition we can now modify the involution $\Theta$ and we get a formula for the product
$ \fS_\alpha \fS_\delta$ that involves many fewer terms then Equation \eqref{expGprodT}:
\begin{Corollary} \label{cor:gen}
\begin{equation}\label{expYprodT}
\fS_\alpha \fS_\delta  = \sum_{\sigma \in S_m} \sum_{{sh(T)=\gamma/\alpha \atop c(T)
={\delta+\sigma-Id}}\atop (\sigma,T) \text{ I-Yamanouchi}} (-1)^\sigma \fS_\gamma.
\end{equation}
\end{Corollary}
 
Unfortunately, our experimentation show that this formula is not the best possible.
In the product of $\fS_{11}\fS_{113}$ we get two extra $\Theta$ fixed points:
$$   (\sigma=132,\tikztableausmall{{ X,1},{X,3},{2,2},{3}} ) \text{ and } (\sigma= 231, ,\tikztableausmall{{ X,3},{X,1},{1,2},{2}})$$
which cancel (having opposite sign but same $\gamma$ in \eqref{expGprodT}.
Still the fix points depends only on $\ell(\alpha)$ and not $\alpha$ itself.

Lets look at some examples: first easy, $\fS_2 \fS_{1,2}=\fS_{2,1,2}$. The only I-Yamanouchi tableau is
$$
\big((1,2), \tikztableausmall{{X,X},{1},{2,2}}\big)  {\buildrel Y_{(1,2)}\over\longmapsto}
\tikztableausmall{{2},{3,3}}
$$
All other terms cancel. Here are some example of the involution
$$
\tikztableausmall{{X,X, 1,1,2}}  {\buildrel Y_{(2,1)}\over\longmapsto}
\tikztableausmall{{1},{1,1}} {\buildrel \Theta\over\longmapsto}
\tikztableausmall{{1},{1,1}} {\buildrel Y_{(1,2)}^{-1}\over\longmapsto}
\tikztableausmall{{ X, X,1,2,2}}
$$
$$
\tikztableausmall{{X,X, 1},{2,2}}  {\buildrel Y_{(1, 2)}\over\longmapsto}
\tikztableausmall{{1},{2,2}} {\buildrel \Theta\over\longmapsto}
\tikztableausmall{{1},{2,2}} {\buildrel Y_{(2, 1)}^{-1}\over\longmapsto}
\tikztableausmall{{ X, X,2},{1,1}}
$$
 
As a  second example consider
$$\fS_2 \fS_{2,4}=\fS_{3,1,4}+\fS_{2,2,4}+\fS_{3,2,3}-\fS_{5,3}-\fS_{4,3,1}$$
The Yamanouchi tableaux are
$$
\big((1,2), \tikztableausmall{{X,X,1},{1},{2,2,2,2}}  \big); \ \big((1,2),  \tikztableausmall{{X,X},{1,1},{2,2,2,2}}   \big); \ \big((1,2),  \tikztableausmall{{X,X,1},{1,2},{2,2,2}}   \big); $$
$$ \big((2,1),\tikztableausmall{{X,X,2,2,2},{1,1,1}}  \big); \ \big((2,1),  \tikztableausmall{{X,X,2,2},{1,1,1},{2}}  \big); \ $$
Some example of the involution in action:
$$
\tikztableausmall{{X,X},{1,1,2},{2,2,2}}  {\buildrel Y_{(1,2)}\over\longmapsto}
\tikztableausmall{{2,2},{2,3,3,3}} {\buildrel \Theta\over\longmapsto}
\tikztableausmall{{3,3,3},{2,2,2}} {\buildrel Y_{(2,1)}^{-1}\over\longmapsto}
\tikztableausmall{{X,X},{1,1,1},{2,2,2}}  
$$
$$
\tikztableausmall{{X,X,1},{1,2,2,2},{2}}  {\buildrel Y_{(1,2)}\over\longmapsto}
\tikztableausmall{{1,2},{2,2,3,3}} {\buildrel \Theta\over\longmapsto}
\tikztableausmall{{1,2,3},{2,2,2}} {\buildrel Y_{(2,1)}^{-1}\over\longmapsto}
\tikztableausmall{{X,X,2},{1,1,1,2},{2}}  
$$
$$
\tikztableausmall{{X,X,1,1},{2,2,2,2}}  {\buildrel Y_{(1,2)}\over\longmapsto}
\tikztableausmall{{1,1},{2,2,2,2}} {\buildrel \Theta\over\longmapsto}
\tikztableausmall{{1,1,2},{2,2,2}} {\buildrel Y_{(2,1)}^{-1}\over\longmapsto}
\tikztableausmall{{X,X,2,2},{1,1,1,2}}  
$$
This last example shows some of the problems we run into trying to define the I-Yamanouchi conditions directly.
From the proof of Theorem~\ref{ImmLRrule} we see that the only time $\theta_r(Y_\sigma(T))$ is not in the image of $Y_{(r\,\,r+1)\circ \sigma}$
is if $b'$ or $a$ correspond to an element in the first column of $T$. This is a necessary condition but it not a sufficient condition as shown
with $\Theta^{-1}$ just above.
 
 \section{classical Littlewood-Richardson rule}
In this section, we use Theorem~\ref{ImmLRrule} to deduce the classical Littlewood-Richardson rule in Theorem~\ref{LRrule}.
This is particularly interesting in light of Corollary~\ref{cor:newsym}
which give a new relation among the non-commutative Littlewood-Richardson coefficients.
This expression involves a fixed number of terms for all partitions of the same length.  
In particular, the complexity of computing the left Pieri rule $\fS_{\op1^n\cl}\fS_\lambda$
is as great as computing a general Littlewood-Richardson coefficient.
One can now understand why the left Pieri rule has to be much more complex than the right Pieri rule.
 
We now show combinatorially the identity in Corollary~\ref{ImmLRgivesLR}.
The left side of the identity counts the number of Yamanouchi  semi-standard  tableaux of shape $\nu / \mu$ and content $\lambda$.
The right side of the identity counts (with sign) the number of Yamanouchi immaculate tableaux of shape
$(\nu\star\sigma) / \mu$ and content $\lambda$.
This can be represented as
\begin{equation}\label{LRexpinImm}
c_{\mu,\lambda}^\nu = \sum_{sh(T)=\nu/\mu, \ semi-standard \atop c(T)=\lambda, \  
Yamanouchi} \!\!\!\!1 \quad = \  \sum_{\sigma \in S_m} \sum_{sh(T)=(\nu\star\sigma)/\mu,
\ immaculate \atop c(T)=\lambda,\ Yamanouchi} (-1)^\sigma .
\end{equation}
 
A combinatorial proof of this  identity is again by constructing a sign reversing involution
$\Psi$ on the set of pair $(\sigma,T)$, where $T$ is an immaculate Yamanouchi tableau of
$sh(T)=(\nu\star\sigma)/\mu$ and $c(T)=\lambda$. Unlike Section~\ref{sec:LRimm}, the involution
must now act on the shape of $T$ instead of its content.
The desired involution must preserve the content and the property that the tableau is immaculate Yamanouchi.
As before, we remark that $\nu\star\sigma$ is a partition if and only if $\sigma$ is the identity.
Furthermore, an  immaculate Yamanouchi tableau $T$ of $sh(T)=(\nu\star\sigma)/\mu$ and $c(T)=\lambda$
is semi-standard if and only if $\sigma=Id$ and $T$ is column strict. These are the desired fixed points of our involution.
This time, we work on $T$ directly (we do not need to introduce a secondary tableau
$Y$ as in Section~\ref{sec:LRimm}). Given a pair $(\sigma,T)$ as above,
we scan the column of $T$ to find the left-most column that contains a nefarious  pattern
 $$ \tikztableausmall{{ X},{$a$}} \qquad \text{ or }\quad \tikztableausmall{{$b$},{$a$}},$$
where $b\ge a$ and the cell above $a$ {\sl is not a cell of $\mu$}.  If there is no such pair,
$T$ is semi-standard and we must have $\sigma=Id$. In this case we have a fixed point $\Psi(Id,T)=(Id,T)$.
Now, if such column exists, let $r,r+1$ be the two rows where the lowest position of such pair (greatest values of $r$).
Remark that by the definition of immaculate tableau, this nefarious  pattern {\sl cannot} be in the first column of $T$.
To define $\Psi$ we use a different procedure than in Section~\ref{sec:LRimm}.
 
Given a pair $(\sigma,T)$ in the sum~\eqref{ImmLRgivesLR}, fix $r,r+1$ as above. Let
$$\nu\star\sigma=(\gamma_1,\ldots,\gamma_r,\gamma_{r+1},\ldots,\gamma_m).$$
We want $\Psi(\sigma,T)=(\sigma\circ(r\,r+1),T')$ where $T'$ is immaculate Yamanouchi
$sh(T')=(\nu\star\sigma\star (r\,r+1))/\mu$ and $c(T)=\lambda$.
Here, $\nu\star\sigma\star(r\,r+1)=(\gamma_1,\ldots,\gamma_{r+1}-1,\gamma_r+1,\ldots,\gamma_m)$.
We have two cases to consider depending if $\gamma_r<\gamma_{r+1}$ or $\gamma_r\ge\gamma_{r+1}$.
In the case $\gamma_r<\gamma_{r+1}$, to get $T'$ with the appropriate shape and content,
we have to move $\gamma_{r+1}-\gamma_r-1$ cells of $T$ from row $r+1$ to row $r$.
In the case $\gamma_r\ge\gamma_{r+1}$, we have to move $\gamma_r-\gamma_{r+1}+1$ from row $r$ to row $r+1$.
There is always at least one cell that moves. This is a consequence of the fact that
if $\gamma\star(r\, r+1)=\gamma$, then $s_\gamma=-s_\gamma=0$ and it is not possible that
$\nu=\gamma\star\sigma^{-1}$.  Thus $|\gamma_{r+1}-\gamma_r-1|>0$.
 
Let us start with the case where $\gamma_r<\gamma_{r+1}$. Since $\mu$ is a partition
we have that the first cell of $T$ is rom $r$ is weakly to the right of the first cell of $T$ in row $r+1$. So the rows $r,r+1$ of $T$ looks like
$$ \tikztableausmall{{ X,X,X,$b_1$,$\scriptstyle \cdots$,$b_s$,$b$,$d_1$,$\scriptstyle \cdots$,$d_t$},
     {$c_1$,$\scriptstyle \cdots$,$c_k$,$a_1$,$\scriptstyle \cdots$,$a_s$,$a$,$e_1$,
     $\scriptstyle \cdots$,$e_t$,$f_1$,$\scriptstyle \cdots$,$f_p$}},$$
where $p=\gamma_{r+1}-\gamma_r>1$,  $k,s,t\ge 0$,
$b_i<a_i$ and  $b$ with $b\ge a$ may not be present when $t=0$. We now shift the row $r$ to the right, until all column are strictly increasing
from $r$ to $r+1$. We have to move at least $p$ cells to the right so that the result is a two row semi-standard tableau. For example
$$ \tikztableausmall{{ X,X,X,1,1,2,5,5,6},   {1,1,2,2,3,3,4,5,6,7,8}}  \mapsto  
       \tikztableausmall{{ X,X,X,X,X,1,1,2,5,5,6},    {1,1,2,2,3,3,4,5,6,7,8}}.$$
Here, $b=5$ is in column $7$ and we have to move the top row two cells to the right before we get a column strict tableau.
In general it will look like
$$ \tikztableausmall{{ X,X,X,X,X,$b_1$,$\scriptstyle \cdots$,$b_s$,$b$,$d_1$,$\scriptstyle \cdots$,$d_t$},
     {$x_1$,,$\scriptstyle \cdots$,, $x_u$,$y_1$,,$\scriptstyle \cdots$,,$y_v$}}.$$
This is a semi-standard skew tableau with two rows.
We now perform $p-1=\gamma_{r+1}-\gamma_r-1>0$ classical jeu-de-taquin using the unique inner cell each time
(see~\cite{Sagan} for a full description of jeu-de-taquin).
This will move $\gamma_{r+1}-\gamma_r-1$ cells from row $r_1$ to row $r$ as desired. In our example above,
$p=2$ so we have to do a single jeu-de-taquin and we get
$$\tikztableausmall{{ X,X,X,X,{\boldentry X},1,1,2,5,5,6},    
{1,1,2,2,3,3,4,5,6,7,8}} \mapsto   \tikztableausmall{{ X,X,X,X,1,1,2,5,5,5,6},    {1,1,2,2,3,3,4,6,7,8}}
       .$$
 
The row $r$ was shifted at least $p$ cell to the right. This means that the entries $b_1,b_2,\ldots,b_s$ are always strictly
smaller then the entry above them is all possible $p-1$ jeu-de-taquin moves.
Hence the entries $b_1,b_2,\ldots,b_s$ will slide to the left $p-1$ cells
as we perform $p-1$ jeu-de-taquin. This will leave unchanged all the cells
that are weakly to the left of the cell with $b_s$. In particular all the entries
weakly to the left of $a$ in row $r+1$ will no move. The tableau $T'$
is obtain by sliding back the new row $r$ to the original position given by $\mu_r$.
The row $r,r+1$ of $T'$ will now look like
$$ \tikztableausmall{{ X,X,X,$b_1$,$\scriptstyle \cdots$,$b_s$,$b'$,$d'_1$,,$\scriptstyle \cdots$,,$d'_q$},
     {$c_1$,$\scriptstyle \cdots$,$c_k$,$a_1$,$\scriptstyle \cdots$,$a_s$,$a$,$e'_1$,$\scriptstyle \cdots$,$e'_j$}}.$$
If $b'$ is not $b$, then it was replaced by a number that appeared to the right of $a$ in the original row $r+1$. Hence $b'\ge a$.
Since no entries strictly to the left of $a$ changed between $T$ and $T'$ we have that $b'\ge a$ is still the  pair we will find in $T'$ according
to the choice of $r,r+1$. The tableau $T'$ has the desired shape and content.
It is an immaculate tableau because the pair $b'\ge a$ is not in the first column.
Hence the first column is strict. The fact that $T'$ is still immaculate will follow from  
Lemma~\ref{lem:jdtyama} bellow. In our running example $T\mapsto T'$ gives
$$ \tikztableausmall{{ X,X,X,1,1,2,5,5,6},   {1,1,2,2,3,3,4,5,6,7,8}}  \mapsto  
       \tikztableausmall{{ X,X,X,1,1,2,5,5,5,6},    {1,1,2,2,3,3,4,6,7,8}}.$$
 
Now the case where $\gamma_r\ge\gamma_{r+1}$. The tableau $T$ in row $r,r+1$ will now look like
$$ \tikztableausmall{{ X,X,X,$b_1$,$\scriptstyle \cdots$,$b_s$,$b$,$d_1$,$\scriptstyle \cdots$,$d_t$,$f_1$,$\scriptstyle \cdots$,$f_p$},
     {$c_1$,$\scriptstyle \cdots$,$c_k$,$a_1$,$\scriptstyle \cdots$,$a_s$,$a$,$e_1$,$\scriptstyle \cdots$,$e_t$}},$$
where $p=\gamma_r-\gamma_{r+1}\ge 0$, $k,s,t\ge 0$, $b_i<a_i$ and  $b\ge a$ must be present. In order to reverse the step above,
the row $r$ of $T$ is first shifted at least 1 cell to the right (since $b\ge a$).
In particular, it will be possible to perform $p+1$ jeu-de-taquin inverse.
As above, the entry weakly to the left of $b_s$ will not change in row $r$ and the entry weakly to the left of $a$ will not change in row $r+1$.
The resulting tableau $T'$ obtain after the $p+1$ jeu-de-taquin inverse  and shifted back,
will have the required shape and content and will be immaculate.
The Yamanouchi condition follows from Lemma~\ref{lem:jdtyama} below.  This conclude our combinatorial proof of Equation~\eqref{LRexpinImm}
 
\begin{Lemma}\label{lem:jdtyama}
If $T$ is a Yamanouchi skew semi-standard tableau, then jeu-de-taquin or jeu-de-taquin inverse preserve the Yamanouchi condition.
\end{Lemma}
 
\todo{NB: I think this is a well known results. Now that I state it this way, I believe it should be in the work of Lascoux-Schutz.
This has to do with the fact that Yamanouchi tableaux rectify (with full jeu-de-taquin) to the super-standard tableau.
Isn't that how Littlewood-Richardson is shown using jeu-de-taquin?? Anyway, I can write a proof, but if we find it in the literature it
would save some paper...}
 
\begin{Remark} The involution $\theta_r$ of Section~\ref{sec:LRimm} would not work to define
$\Psi$ as it does not preserve the Yamanouchi condition.
It may be possible to modify $\theta_r$ so that the two involution $\Theta$ and $\Psi$ acting on the
$Y_\sigma(T)$ and $T$ respectively looks similar,
but this would have complicated the proof in Section~\ref{sec:LRimm}.
\end{Remark} 

\end{moredetails}

\section{A geometric interpretation of the immaculate Littlewood-Richardson coefficients }\label{niceass}
 
There are several geometric interpretations of the Littlewood-Richardson coefficients
$c_{\mu, \lambda}^\nu$ as the number of integral points inside a certain polytope (see for instance \cite{BZ}, \cite{GZ}, \cite{PakVallejo}).
In this section we construct a polytope $\polytope$ that would give rise to an analogous interpretation of the immaculate Littlewood-Richardson coefficients $C_{\alpha,\nu}^{\beta}$, where $\alpha$ and $\beta$ are compositions, and $\nu$ is a partition.

Let $m$ be the maximum of the three integers $\ell(\alpha),\ell(\beta)$ and $\ell(\nu)$. $\polytope$ will live in an $m+2 \choose 2$ dimensional vector space with coordinates $a_{ij}$ where $0 \leq i \leq j \leq m$ and $a_{00} = 0$.

We construct $\polytope$ in a fashion similar to that of \cite{PakVallejo}. If necessary, extend $\alpha$, $\beta$, and $\nu$ with $0$'s at the end so that they all have length $m$. Each point in $\polytope$ will have for coordinates, the integers in the following array, subject to certain linear inequalities.
  $$\begin{array}{ccccccccc}
 &&&&a_{00} &&& \\
 &&&a_{01}&&a_{11}&&\\
 &&a_{02}&&a_{12}&&a_{22}\\\\
 &\udots&&\ddots&&\udots&&\ddots\\\\
 a_{0m}&&a_{1m}&&\cdots&&a_{m-1\,m}&&a_{mm}\\
 \end{array}$$
 
The linear inequalities defining $\polytope$ are the following.
Let $N=|\nu|$, then
\begin{enumerate}
\item \label{C1} $a_{ij}\ge 0$ for $1 \leq i \leq j \leq m$
\item \label{C2} $a_{0j}=\alpha_j$ for $1 \leq i \leq m$
\item \label{C3} $\sum_{p=0}^j  a_{pj} = \beta_j$ for $1 \leq j \leq m$
\item \label{C4} $\sum_{q=i}^m a_{iq} = \nu_i$ for $1 \leq i \leq m$
\item \label{C5} $\sum_{q=i}^{j} a_{iq} -  \sum_{q=i+1}^{j} a_{i+1\,q}\ge 0$ for $1 \leq i \leq j \leq m$
\item \label{C6} $\sum_{p=0}^{i-1} N a_{pj} -  \sum_{p=0}^{i} a_{p\,j+1}\ge 0$ for $1\leq i \leq j < m$
\end{enumerate}
 
These equations define a polytope because they are the intersection of planes and half planes.  By comparison with
the Littlewood-Richardson coefficients, if $\alpha = \mu$, $\beta = \lambda$ where $\mu$, $\lambda$ and $\nu$ are
partitions such that $|\lambda| = |\mu| + |\nu|$, then \cite[Lemma 3.1]{PakVallejo} states that $c_{\mu,\nu}^\lambda$
is equal to the number of integer points satisfying conditions \eqref{C1} through \eqref{C5} and the inequality
 
$$\hbox{(CS) \hskip .2in}\sum_{p=0}^{i-1} a_{pj} - \sum_{p=0}^i a_{p\,j+1} \geq 0 \hbox{ for all } 1 \le i \le j < m$$
in place of condition \eqref{C6}.
 
\begin{Theorem}\label{THM:geom1} For $\alpha, \beta$ compositions and $\nu$ a partition all of length less than
or equal to $m$ such that $|\beta| = |\alpha| + |\nu|$, $C_{\alpha, \nu}^{\beta}$ is equal to the number of
integer points satisfying conditions \eqref{C1}--\eqref{C6}.
\end{Theorem}
 
\begin{proof}
We want to show that there is a bijection between the integral points inside the region
defined by conditions \eqref{C1} through \eqref{C6} and Yamanouchi
immaculate tableaux of shape $\beta/\alpha$ and content $\nu$. Fix a Yamanouchi immaculate tableau $T$
of shape $\beta/\alpha$ and content $\nu$.
Let  $A(T)=(a_{ij})_{0 \leq i\leq j\leq m}$ be the point defined by
\begin{enumerate}[(i)]
\item \label{Ci} $A(T)_{00}=0$
\item \label{Cii} $A(T)_{0j}=\alpha_j$ for $1 \leq j \leq m$
\item \label{Ciii} $A(T)_{ij}$ is the number of entries $i$ in  row $j$ of $T$, for $1\le i\le j\le m$.
\end{enumerate}
 
We can check that $A(T)$ is in the region defined by \eqref{C1}--\eqref{C6}.
Condition \eqref{Ciii} implies that condition \eqref{C1} will hold and condition \eqref{Cii} is a statement that \eqref{C2}
will be true.  The sum of the entries $A(T)_{pj}$ for $0 \le p \le j$ is equal to the number of entries in row $j$
and hence will sum to $\beta_j$ (hence \eqref{C3} will hold).  The sum of the entries $A(T)_{iq}$ for $i \le q \le m$ is
the number of entries of $T$ equal to $i$ and this is equal to $\nu_i$ which implies~\eqref{C4}.
 
The sum $\sum_{q=i}^{j} A(T)_{iq}$ represents the number of entries $i$ in rows $i$ through $j$ 
in $T$.  If $T$ is Yamanouchi, then the number of labels $i$ in rows $i$ through $j$ is larger than or equal to the number of
labels $i+1$ in rows $i+1$ through $j$.  This implies the inequality
$\sum_{q=i}^{j} A(T)_{iq} \ge \sum_{q=i+1}^{j} A(T)_{i+1\,q}$ and \eqref{C5} holds.
 
Assume that $T$ is not immaculate column strict,
then there will be a row $j$ where $\alpha_j = \alpha_{j+1}=0$
and the first non-zero entry in row $j+1$ is smaller than or equal to the first non-zero entry in row $j$.  
If the first non-zero entry in row $j+1$ is $i$, then $\sum_{p=0}^{i} A(T)_{p\,j+1}$
will be non-zero and $\sum_{p=0}^{i-1} A(T)_{pj}$ will be $0$, hence \eqref{C6} will not hold.  
 
Now assume that $T$ is immaculate column strict,
then every row either has $A(T)_{0j} = \alpha_j>0$
or $\alpha_j = \alpha_{j+1}=0$ and the smallest entry in row $j$ is smaller than the smallest entry in row $j+1$.
If $\alpha_j>0$, then because $N$ is chosen to be large,
$\sum_{p=0}^{i-1} N A(T)_{pj} \gg \sum_{p=0}^{i} A(T)_{p\,j+1}$ for all $i$ such that $1 \leq i \leq j$.
If $\alpha_j = \alpha_{j+1}=0$ and the smallest entry in row $j$ is $i'$, then
$\sum_{p=0}^{i-1} N A(T)_{pj} - \sum_{p=0}^{i} A(T)_{p\,j+1}=0$ if $i\le i'$ and
$\sum_{p=0}^{i-1} N A(T)_{pj} \gg \sum_{p=0}^{i} A(T)_{p\,j+1}$ if $i> i'$.  In all of these cases \eqref{C6}  holds.

Given a point $(a_{ij})$ which satisfies conditions \eqref{C1}--\eqref{C6}, let $T$ be the skew tableau of
with inner shape $\op a_{01}, a_{02}, \ldots, a_{0m}\cl = \alpha$ and with $a_{ij}$ entries
$i$ in row $j$ for $1 \leq i < j \leq m$ such that these entries are arranged so they are weakly increasing in the rows.  
The $j^{th}$ row of $T$ has a total of $\sum_{p=0}^j a_{pj} = \beta_j$ cells.  The number of entries $i$ will be
$\sum_{q=i}^m a_{iq} = \lambda_i$.  The argument above shows that Condition \eqref{C5} is equivalent to
the tableau is Yamanouchi and Condition \eqref{C6} implies that $T$ is strictly increasing
in the first column.
\end{proof}
 
\begin{Example} \label{ex:smallarray}
The two Yamanouchi immaculate tableaux from Example \ref{ex:smallex} have arrays $A(T)$ represented by the following triangles.
 
$$
\begin{array}{ccccccc}
&&&0&&&\\
&&1&&2&&\\
&2&&1&&1&\\
0&&0&&1&&1\\
\end{array}
\hskip .3in
\begin{array}{ccccccc}
&&&0&&&\\
&&1&&2&&\\
&2&&0&&2&\\
0&&1&&0&&1\\
\end{array}
$$
\end{Example}
 
\begin{Example} \label{ex:largerex}
A larger example of a skew immaculate Yamanouchi tableau with $\alpha = \op 3,2,3,1\cl$, $\beta = \op 3,5,5,5,2\cl$ and $\lambda = \op4,3,3,1\cl$
is represented by the following
$$ T= \tikztableausmall{{\boldentry X,\boldentry X,\boldentry X},{\boldentry X, \boldentry X, 1,1,1},{\boldentry X,\boldentry X,\boldentry X,1,2},{\boldentry X,2,2,3,3},{3,4}}.$$
The array for this tableau is
$$A(T) = \begin{array}{ccccccccccc}
&&&&&0&&&&&\\
&&&&3&&0&&&&\\
&&&2&&3&&0&&&\\
&&3&&1&&1&&0&&\\
&1&&0&&2&&2&&0&\\
0&&0&&0&&1&&1&&0\\
\end{array}$$
\end{Example}
 
We now consider the linear transformation where
 $$h_{ij} = \sum_{p=0}^i \sum_{q=p}^j  a_{pq}.$$
We refer the point resulting of this transformation the {\it hive array}
corresponding to the point or tableau.
This transformation sends the region in Theorem~\ref{THM:geom1} into the following region
cut out by the intersection of the inequalities:
\begin{enumerate}[(1$'$)]
\item \label{Cp1} $h_{0j}-h_{0\,j-1}=\alpha_j$ for $1 \leq j \leq m$
\item \label{Cp2} $h_{jj}-h_{j-1\,j-1} = \beta_j$ for $1 \leq j \leq m$
\item \label{Cp3} $ h_{im} - h_{i-1\,m} = \nu_i$ for $1 \leq i \leq m$
\item \label{Cp4} $ h_{ij} - h_{i\,j-1}  \ge h_{i-1\,j}-h_{i-1\,j-1}$ for $1 \leq i < j \leq m$
\item \label{Cp5} $  h_{ij}-h_{i-1\,j} \ge h_{i+1\,j+1} - h_{i\,j+1}$ for $1 \leq i \leq j < m$
\item \label{Cp6} $N( h_{i-1\,j}-h_{i-1\,j-1})\ge h_{i\,j+1}-h_{ij}$ for $1 \leq i \leq j < m$
\end{enumerate}
 
\begin{Corollary}\label{cor:hives}  Equations (\ref{Cp1}$'$)--(\ref{Cp6}$'$)
determine a polytope $\polytope$ such that the number of integral points inside the region is equal to  $C_{\alpha, \nu}^{\beta}$.
\end{Corollary}
 
The polytope defined by the equations (\ref{Cp1}$'$)--(\ref{Cp6}$'$)
is an analogue of the hive polytope of Knutson and Tao \cite{KnTao}.
 
\begin{Example}
The two Yamanouchi immaculate tableaux from Example \ref{ex:smallex} and \ref{ex:smallarray} have the following hive arrays.
$$
\begin{array}{ccccccc}
&&&0&&&\\
&&1&&3&&\\
&3&&6&&7&\\
3&&6&&8&&9\\
\end{array}
\hskip .3in
\begin{array}{ccccccc}
&&&0&&&\\
&&1&&3&&\\
&3&&5&&7&\\
3&&6&&8&&9\\
\end{array}
$$
\end{Example}
 
\begin{Example} The tableau $T$ from Example \ref{ex:largerex} has the following hive array.
$$\begin{array}{ccccccccccc}
&&&&&0&&&&&\\
&&&&3&&3&&&&\\
&&&5&&8&&8&&&\\
&&8&&12&&13&&13&&\\
&9&&10&&16&&18&&18&\\
9&&10&&16&&19&&20&&20\\
\end{array}$$
\end{Example}


\end{document}